\documentclass[11pt]{amsart}
\usepackage{amsmath,amsthm}%,amsfonts}
\usepackage {latexsym}
\usepackage{amssymb}
\usepackage{xcolor}

\RequirePackage[normalem]{ulem} %DIF PREAMBLE
\RequirePackage{color}\definecolor{RED}{rgb}{1,0,0}\definecolor{BLUE}{rgb}{0,0,1} %DIF PREAMBLE
\usepackage[utf8]{inputenc}
\usepackage{xspace}
\usepackage[colorlinks=true]{hyperref}

\newcommand{\Sph}{\mathbb{S}}

\newcommand{\R}{\mathbb{R}}

\newcommand{\Z}{\mathbb{Z}}

\definecolor{darkgreen}{rgb}{0,0.7,0}%darkgreen
\newcommand{\jm}[1]{\textcolor{darkgreen}{}}

\newtheorem{thm}{THEOREM}[section]
\newtheorem{remark}[thm]{REMARK}
\newtheorem{lem}[thm]{LEMMA}
\newtheorem{defn}[thm]{DEFINITION}
\newtheorem{prop}[thm]{PROPOSITION}

\newtheorem{conj}[thm]{Conjecture}

\date{\today}

\begin{document}
\title[Global controllabilty/stabilization  of the wave maps equation from $\Sph^1$ to $\Sph^k$]{Global controllability and stabilization of the wave maps equation from a circle to a sphere}

\author{Jean-Michel Coron}
\address{Sorbonne Universit\'{e}, Universit\'{e} Paris-Diderot SPC, CNRS, INRIA, Laboratoire Jacques-Louis Lions, LJLL,  \'{e}quipe CAGE, F-75005 Paris, France}
\email{\texttt{jean-michel.coron@sorbonne-universite.fr}}
\thanks{}

%    author two information

\author{Joachim Krieger}
\address{B\^{a}timent des Math\'{e}matiques, EPFL\\Station 8, CH-1015 Lausanne, Switzerland}
\email{\texttt{joachim.krieger@epfl.ch}}
\thanks{}

%    author three information
\author{Shengquan Xiang}
\address{School of Mathematical Sciences, Peking University, 100871, Beijing, P. R. China}
\email{\texttt{shengquan.xiang@math.pku.edu.cn}}
\thanks{}

\begin{abstract}
 Continuing the investigations started in the recent work \cite{Krieger-Xiang-2022} on  semi-global controllability and stabilization of the $(1+1)$-dimensional  wave maps equation with spatial domain $\mathbb{S}^1$ and target $\mathbb{S}^k$,
where {\it semi-global} refers to the $2\pi$-energy bound, we prove  global exact controllability of the same system for $k>1$ and  show that the $2\pi$-energy bound is a strict threshold for uniform asymptotic stabilization via continuous time-varying feedback laws indicating that the damping stabilization in \cite{Krieger-Xiang-2022} is sharp.  Lastly, the global exact controllability for $\mathbb{S}^1$-target within minimum time is discussed. 
\end{abstract}
\maketitle

\section{Introduction}
 Recently, control problems of the geometric wave maps equations were  studied in \cite{Krieger-Xiang-2022}; it appears that controllability of this particular model has not been considered before. Recall that by {\it controllability}, one means that for any given initial and final states, one can construct a localized control that steers the solution from the one state to the other, and that by {\it stabilization}, one means that by using a suitable localized control feedback, depending on the state at the current time but not on the initial data, one can stabilize the system. In this paper, we continue to investigate the global controllability and stabilization problems of this geometric model for the case of the spatial domain $\mathbb{S}^1$ and $\mathbb{S}^k$ target.  This research enables us to discover more general and intrinsic control properties of this geometric model including:
\begin{itemize}
    \item The damped wave maps flow converges to harmonic maps with quantitative asymptotic analysis.  One  can  naturally compare it to the convergence of the harmonic maps heat flow to harmonic maps.
    \item Quantitative global exact controllability of the controlled wave maps equations when $k\geq 2$.  One has to bypass the stationary states of the system, i. e. the harmonic maps,  using well-designed controls.
    \item $2\pi$-energy level is not only a limitation of the damping stabilization caused by  harmonic maps,  but also a generic  obstruction for general  uniform asymptotic stabilization, for all $k$. 
 \end{itemize}

The geometric wave maps equations generalize the wave equations taking values in $\mathbb{R}$ to those taking values in geometric targets, and more specifically Riemannian manifolds. Let us be given a Riemannian manifold $(\mathcal{M}, g)$ and the flat space endowed with Minkowski metric $(\mathbb{R}^{1+n}, h)$. The functions $\phi: \mathbb{R}\times\mathbb{R}^n\rightarrow \mathcal{M}$ that are critical for the Lagrangian action functional
\begin{equation*}
    L^h_{\mathcal{M}}= \int_{\mathbb{R}^{1+n}} -|\partial_t \phi|_g^2+ |\nabla_x \phi|_g^2 dx dt
\end{equation*}
satisfy the following system of geometric wave equations in local coordinates:
\begin{equation*}
    \Box \phi^i+ \Gamma^i_{jk} \partial^{\alpha} \phi^j \partial_{\alpha} \phi^k=0\;\; \textrm{ where } \; \Box:= -\partial_{tt}+ \Delta,
\end{equation*}
which is an example of a system of semilinear wave equations.
We refer to the survey paper  by  Tataru  \cite{Tataru-2004} for an excellent introduction to this topic, and to the former paper \cite{Krieger-Xiang-2022} by the last two authors for first investigations on the controllability for this equation.  \\

Given a Riemannian target manifold, we are interested in the global controllability and stabilization problems of the corresponding wave maps equation. Heuristically speaking, one can naturally expect to obtain local controllability and stabilization properties for  wave maps equations  from  the related results on linear wave equations, a well-known fact (see for example \cite{Bardos-Lebeau-Rauch}). We  mainly focus on {\bf global} control problems, or {\bf semi-global} control problems in the presence of obstructions; for example this is the case for the damping stabilization where harmonic maps prevent us from getting global stabilization.

To reduce the technical difficulty of the problem, we start by considering the simplest case that the maps are from $\mathbb{R}\times\mathbb{S}^1$ to  $\mathbb{S}^k$-target, where the analysis is somewhat easier thanks to the fact that the geometric target  $\mathbb{S}^k$, the unit sphere of $\R^{k+1}$, is simple (in particular has explicit geodesics), and leave the (much) more complicated cases that map $\mathbb{R}\times\mathbb{S}^l$ to $\mathcal{M}$ to follow up work. We note right away  that in the case $l>1$,  more complex phenomena will have to be dealt with, most notably  finite time blow up  \cite{Krieger-Schlag-Tataru, R-R-2012}.  In the specific case that we study in this paper,  the wave maps equation from  $[0, T)\times\mathbb{S}^1$ to $\mathbb{S}^k$ can be written as follows:
\begin{equation*}
\begin{cases}
     \Box \phi=  \left(|\phi_t|^2- |\phi_x|^2\right)\phi, \; \forall (t, x)\in (0, T)\times \mathbb{S}^1, \\ (\phi, \phi_t)(0)= (g_0, g_1).
     \end{cases}
\end{equation*}

Let us first recall the control framework introduced in \cite{Krieger-Xiang-2022}. For any initial state $(g_0, g_1): \Sph^1\rightarrow T\Sph^k$ and  any $\mathbb{R}^{k+1}$-valued source term $f(t, x)$ having sufficient regularity  we consider the following forced nonlinear wave equation:
\begin{equation}\label{eq:inhomowavemaps}
\begin{cases}
         \Box \phi= \left(|\phi_{t}|^2- |\phi_{ x}|^2\right) \phi+
     f^{\phi^{\perp}}, \; \forall (t, x)\in (0, T)\times \mathbb{S}^1, \\  (\phi, \phi_t)(0)= (g_0, g_1),
     \end{cases}
\end{equation}
where by $f^{\phi^{\perp}}$ we refer to the orthogonal projection of $f$ onto the tangent space $T_{\phi}\mathbb{S}^k$ of $\mathbb{S}^k$ at $\phi$: $f^{\phi^{\perp}}=f- \langle f, \phi\rangle \phi$.
When the source term $f(t, x)$ is chosen to  be supported in $[0,T]\times \omega$ for some given nonempty open set $\omega\subset \mathbb{S}^1$, we call  the preceding system {\it{internally controlled}} and the subdomain $\omega$ the control area. \jm{Before there was ``More precisely, the  control term is given \textit{implicitly} by ${\bf 1}_{\omega} f^{\phi^{\perp}}$. With a slight abuse of the terminology we  also call this $\mathbb{R}^{k+1}$-valued  source function $f(t, x)$  the control term of the system''. However the support of ${\bf 1}_{\omega} f^{\phi^{\perp}} $ is not included in $\omega$ but in $\bar \omega$.}

We are also interested in the following {\it  damped wave maps equation} with a view towards the issue of stabilization, where the damping term $a(x) \phi_t$ that is localized in $\omega$ can be regarded as a special control term (also called a  closed-loop stabilization term; moreover, one easily checks that $a(x) \phi_t(t,x)$ belongs to $T_{\phi(t,x)}\mathbb{S}^k$).
\begin{gather}\label{eq:dwm}
     \Box \phi= \left(|\phi_t|^2- |\phi_x|^2\right)\phi+ a(x) \phi_t,
     \end{gather}
throughout this paper the continuous function $a(x)$ is fixed and satisfies:
\begin{equation}
     a(x)\geq 0 \textrm{ in } \mathbb{S}^1,\; \textrm{ supp }a\subset \omega, \; a\not =0.
\end{equation}

Throughout this paper, we  denote by $E$ the energy of the system:
\begin{equation} \label{def-energy}
    E((\phi, \phi_t)(t, \cdot)):= \int_{\mathbb{S}^1} \left(|\phi_x|^2+ |\phi_t|^2\right)(t, x) dx, \; \forall t>0.
\end{equation}
By slightly abusing the notations when there is no risk of confusion, we also denote $E((\phi, \phi_t)(t, \cdot))$
by $E(\phi[t])$ or simply $E(t)$. \\

Direct energy estimates indicate that the damping effect dissipates the energy:
\begin{equation*}
    \frac{d}{dt} E(t)= -2\int_{\Sph^1} a(x)|\phi_t|^2(t, x) dx\leq 0.
\end{equation*}
Moreover, it is further proved in \cite{Krieger-Xiang-2022} that the energy decays exponentially towards  0 provided that the  initial value of the energy is strictly smaller than $2\pi$, see Proposition \ref{thm:semicontrol} $(i)$ for details.
This stabilization technique has been heavily adapted to the study of control problems of dispersive equations including linear and defocusing wave equations, KdV, Schr\"{o}dinger equations, among others. Such an exponential stability property is typically related to the so-called {\it observability inequalities},
\begin{equation*}
  E(0)\leq c \int_0^T\int_{\mathbb{S}^1}a(x)|\phi_t|^2(t, x) dxdt, \; \forall \phi[0].
\end{equation*}

More precisely, the following  results concerning semi-global controllability and stabilization are demonstrated  in \cite{Krieger-Xiang-2022}:
% \begin{thm}\label{thm:semidecay} {\rm (\cite[Theorem 1.3]{Krieger-Xiang-2022}, semi-global stabilization of the damped wave maps)}
% Let $\nu>0$. There exist some effectively computable $C$ and $c$ such that for any initial state $\phi[0]: \mathbb{S}^1\rightarrow \mathbb{S}^k\times T\mathbb{S}^k$  satisfying
% \begin{equation*}
%     E(0)\leq 2\pi- \nu,
% \end{equation*}
% the unique solution of the damped wave maps equation \eqref{eq:dwm}
% decays exponentially:
% \begin{equation*}
%     E(t)\leq C e^{-c t} E(0), \; \forall t\in (0, +\infty).
% \end{equation*}
% \end{thm}

% \begin{thm}\label{thm:semicontrol} {\rm (\cite[Theorem 1.1]{Krieger-Xiang-2022}, semi-global exact controllability of   wave maps)}
% Let  $\nu>0$.  There exist some effectively computable $T\geq 2\pi$ and $C>0$ such that, for any pair of initial and final target states $u[0]$ and $u[T]: \mathbb{S}^1\rightarrow \mathbb{S}^k\times T\mathbb{S}^k$ satisfying
% \begin{equation*}
%     \|u[0]\|_{\dot{H}^1_x\times L^2_x}^2, \|u[T]\|_{\dot{H}^1_x\times L^2_x}^2\leq 2\pi-\nu,
% \end{equation*}
% we can explicitly construct a $\mathbb{R}^{k+1}$-valued control $f(t, x)$ satisfying
% \begin{equation*}
%      \|f\|_{L^{\infty}_t L^2_x([0, T]\times \mathbb{S}^1)}\leq C \left(\|u[0]\|_{\dot{H}^1_x\times L^2_x}+ \|u[T]\|_{\dot{H}^1_x\times L^2_x}\right) ,
% \end{equation*}
% such that the unique solution of the inhomogeneous wave maps equation \eqref{eq:inhomowavemaps} with initial state $u[0]$ satisfies that
%   $(\phi, \phi_t)(t, x)\in \mathbb{S}^k\times T\mathbb{S}^k$ and that
% $ \phi[T]= u[T]$.
% \end{thm}

\begin{prop}\label{thm:semicontrol} {\rm (\cite[Theorem 1.1 and Theorem  1.3]{Krieger-Xiang-2022})}  Let $\nu>0$. Consider the wave maps equation from $\mathbb{R}\times\mathbb{S}^1$ to $\mathbb{S}^k$, $k\geq 1$.\\
(i) \textrm{(Semi-global stabilization of the damped wave maps equation)}
There exist some effectively computable $C$ and $c>0$ such that for any initial state $\phi[0]: \mathbb{S}^1\rightarrow  T\mathbb{S}^k$
satisfying
\begin{equation*}
    E(0)\leq 2\pi- \nu,
\end{equation*}
the unique solution of the damped wave maps equation \eqref{eq:dwm}
decays exponentially:
\begin{equation*}
    E(t)\leq C e^{-c t} E(0), \; \forall t\in (0, +\infty).
\end{equation*}
\jm{Before there was $\mathbb{S}^k\times T\mathbb{S}^k$ instead of $T\mathbb{S}^k$. However, note that $T\mathbb{S}^k=\{(y,z)\in \R^{k+1}\times\R^{k+1}: |y|=1 \text{ and }
y\cdot z=0\}$. The same problem appears at other places.}

 (ii) \textrm{(Semi-global exact controllability of  the wave maps equation)}
 There exist some effectively computable $T\geq 2\pi$ and $C>0$ such that, for any pair of initial and final target states $u[0]$ and $u[T]: \mathbb{S}^1\rightarrow  T\mathbb{S}^k$ satisfying
\begin{equation*}
    \|u[0]\|_{\dot{H}^1_x\times L^2_x}^2, \|u[T]\|_{\dot{H}^1_x\times L^2_x}^2\leq 2\pi-\nu,
\end{equation*}
we can explicitly construct a $\mathbb{R}^{k+1}$-valued control $f(t, x)$ compactly supported in $[0, T]\times \omega$ and satisfying
\begin{equation*}
     \|f\|_{L^{\infty}_t L^2_x([0, T]\times \mathbb{S}^1)}\leq C \left(\|u[0]\|_{\dot{H}^1_x\times L^2_x}+ \|u[T]\|_{\dot{H}^1_x\times L^2_x}\right) ,
\end{equation*}
such that the unique solution of the inhomogeneous wave maps equation \eqref{eq:inhomowavemaps} with initial state $u[0]$ satisfies
$ \phi[T]= u[T]$.
\jm{I think it is better to remove ``$(\phi, \phi_t)(t, x)\in  T\mathbb{S}^k$'' since it seems to say that this property requires specific properties of $f$ which is not the case.}
\end{prop}

Regarding these semi-global control results, where the word ``semi-global" refers to the fact that states are restricted below the first critical energy level $2\pi$, several  questions arise naturally:

{\bf (Q1)} As remarked in \cite{Krieger-Xiang-2022}, the $2\pi$-energy level in Proposition \ref{thm:semicontrol} $(i)$ is optimal in the sense that harmonic maps appear as non-trivial stationary states of the damped wave maps equation. Because harmonic maps (for example an equator) are not minimizers of the energy, they are not stable stationary states.  One may pose the question whether this energy level is a threshold for uniform asymptotic stabilization with arbitrary stationary  feedback law $f(t, x)= F(\phi(t, x), \phi_t(t, x))$ or even time-varying feedback law $f(t, x)= F(t; \phi(t, x), \phi_t(t, x))$.  Specifically, is it possible to construct a continuous time-varying feedback law,
\begin{equation}
\label{feedbacklaw-def}
\begin{array}{ccc}
    F: \mathbb{R}\times H^1_x\times L^2_x(\mathbb{S}^1;T\mathbb{S}^k)&\rightarrow& L^2_x(\mathbb{S}^1) \\
    (t; (\phi, \phi_t))&\mapsto &  F(t; (\phi, \phi_t))
\end{array}
\end{equation}
satisfying
\begin{equation*}
\textrm{supp }F(t; (\phi, \phi_t))\subset \omega, \; \forall t \in \R,\; \forall  (\phi, \phi_t)\in H^1_x\times L^2_x(\mathbb{S}^1;T\mathbb{S}^k)
\end{equation*}
such that the closed-loop system
\begin{equation}\label{closed-loop-timevarying}
     \Box \phi(t, x)= \left(|\phi_{t}|^2- |\phi_{ x}|^2\right)(t, x) \phi(t, x)+  (F(t; (\phi, \phi_t)(t, \cdot))(x))^{\phi(t, x)^{\perp}}
\end{equation}
is uniformly asymptotically stable with a large basin of attraction? To quantify the basin of attraction, we introduce the following definition.
\begin{defn}\label{DEF-usstab}
Let $e>0$.  Let us denote by $\textbf{H}(e)$ the set of  states
    \begin{equation}
        \textbf{H}(e):= \{u[0]: u[0](x)\in T\mathbb{S}^k,\; \forall x\in \mathbb{S}^1, \; E(u[0])\leq e \}.
    \end{equation}
The closed-loop system \eqref{closed-loop-timevarying} is called uniformly asymptotically stable in $\textbf{H}(e)$ if there exists a $\mathcal{KL}$ function $h$, i.e. (see, for example, \cite[Def. 24.2, page 97]{1967-Hahn-book}) a continuous function $h: \mathbb{R}^+\times \mathbb{R}^+ \rightarrow \mathbb{R}^+$ satisfying
\begin{gather*}
  \forall t \in \mathbb{R}^+,\;   h(\cdot,t) \text{ is strictly increasing and vanishes at $0$,}
\\
  \forall s \in \mathbb{R}^+,\;   h(s,\cdot) \text{ is decreasing and } \lim_{t\rightarrow +\infty}h(s,t)=0,
\end{gather*}
 such that the energy of the system decays uniformly as follows:
\begin{equation}\label{estEphit}
    E(\phi[t])\leq h(E(\phi[0]),t),  \; \forall t\in (0, +\infty),\;   \forall \phi[0] \in \textbf{H}(e).
\end{equation}
\jm{May be it is better to use the classical definition of asymptotic stability, which allows to consider functions $h(E(\phi[0]),t)$ which are more general than $h(t)E(\phi[0])$.}
\end{defn}

\begin{remark}
In fact, instead of \eqref{estEphit}, the usual definition of uniform asymptotic stability (see, for example, \cite[Def. 36.9, page 174]{1967-Hahn-book}), requires that, for every $\tau\in \R$,
\begin{equation}\label{estEphit-tau}
    E(\phi[t])\leq h(E(\phi[\tau]),t-\tau),  \; \forall t\in (\tau, +\infty), \; \forall\phi[\tau] \in \textbf{H}(e).
\end{equation}
However \eqref{estEphit} is sufficient for the obstruction given in Theorem \ref{THM-nounifdecay}.
\end{remark}

{\bf (Q2)}  Is it possible to remove the semi-global
restriction in Proposition \ref{thm:semicontrol} $(ii)$ to obtain the stronger global controllability result? Of course, global controllability is impossible when the target is $\mathbb{S}^1$, since the curve $\phi(t)(\cdot): \Sph^1\rightarrow \Sph^1$ has a degree which does not depend on $t$.  However, when the target is a higher dimensional sphere,  harmonic maps are homotopic to  trivial states (any point state). This inspires us to expect such a global controllability result for $k>1$. It is  conjectured in \cite{Krieger-Xiang-2022} that for a general Riemannian manifold target the controllable space (accessible space) is equivalent to its first homotopy class.
\begin{defn}\label{def-con-homo}
    Let $(\mathcal{M}, g)$ be a Riemannian manifold. The controlled wave maps equation $\phi: \mathbb{R}\times \mathbb{S}^1\rightarrow \mathcal{M}$ is said to be {\bf globally exact controllable in homotopy class} if  for any pair of initial and final target  states $(u[0], u[T])$ in  $H^1_x\times L^2_x(\mathbb{S}^1;\mathbb{S}^k)$ both with finite energy and such that $u(0, x)$ and $u(T, x)$ are homotopic, there exists a  control $f_0\in L^2_{t, x}([0, T]\times \mathbb{S}^1)$ with some $T>0$ depending on the given pair of states and having a support included in $[0,T]\times \omega$ such that, the unique solution of the controlled wave maps equation with  initial state $\phi[0]= u[0]$ and  control $f_0$ satisfies $\phi[T]= u[T]$.
\end{defn}
\begin{conj}\label{CONJ}
Let $(\mathcal{M}, g)$ be a Riemannian manifold. The controlled wave maps equation $\phi: \mathbb{R}\times \mathbb{S}^1\rightarrow \mathcal{M}$ is globally exact controllable in   homotopy class.
\end{conj}
This conjecture is the strongest possible one in the sense that as time $t$ changes, $\phi(t)$ always lives in the same homotopy class in $[\mathbb{S}^1,\mathcal{M}]$.\\

{\bf (Q3)}  Finally, observe from Definition \ref{def-con-homo} that the time period $T$ may depend on the initial and final states;  one may further wonder about the possibility of controlling the system  in homotopy class within some  fixed time  period $T$,  or even within short time intervals,  instead of the large-time controllability result presented in the preceding proposition.  Of course,  due to the finite speed of propagation, the minimum time for exact controllability should be at least $2\pi-|\omega|$ (see \cite[Chapter 2.1 and Chapter 2.4]{coron} for a heuristic explanation on this type of minimum time based on transport equations and wave equations).
\\

This paper is devoted to the study of these questions leading to the following answers:

First, we show that $2\pi$-energy is a strict threshold for (time-varying) uniform asymptotic stabilization by means of continuous feedback laws indicating that Proposition \ref{thm:semicontrol} $(i)$ is sharp, also for $k>1$. Its proof is based on a topological observation concerning the evolution of the flow that connects the uniform asymptotic stability and degree theory.  We observe that various topological conditions necessary for uniform asymptotic stabilization of {\it{finite dimensional}} systems are discussed in the literature, for example the criterion based on homotopy groups  found by the first author \cite{Coron-1990}. To the best of our knowledge, this is the first time that such an important topological property appears in the context of stabilization of PDE based models.

\begin{thm}\label{THM-nounifdecay}
   For any  time-varying feedback law F satisfying  conditions $(\mathcal{P}1)-(\mathcal{P}4)$ (the precise conditions will be given in Section \ref{sec:pre})
    \begin{equation*}
      F: \mathbb{R}\times H^1_x\times L^2_x(\mathbb{S}^1;T\mathbb{S}^k)\rightarrow  L^2_x(\mathbb{S}^1),
    \end{equation*}
    the closed-loop system \eqref{closed-loop-timevarying} is not uniformly asymptotically stable in
    $\textbf{H}(2\pi)$.
\end{thm}
\begin{remark}
In the above theorem there is no condition on the support of $F(t,(\phi, \phi_t))$: this obstruction to uniform asymptotic stability holds even for $\omega=\Sph^1$.
\end{remark}

\begin{remark}
It is natural to ask whether for any time-varying feedback law $F$ satisfying conditions $(\mathcal{P}1)-(\mathcal{P}3)$ (namely, we remove the Lipschitz condition concerning the uniqueness of solution to the closed-loop  system)  the closed-loop system \eqref{closed-loop-timevarying} is not uniformly asymptotically stable in
    $\textbf{H}(2\pi)$.
\end{remark}

The second result answers Conjecture \ref{CONJ} for the spherical target case.  We hope that this framework can further inspire a complete proof of the conjecture  for general Riemannian manifold targets.

\begin{thm}[Global exact controllability for $\mathbb{S}^k$-target]\label{THM-globalexact}
Let $k\geq 2$ and $M>0$. There exist some effectively computable $T>0$ and $C>0$, such that for any initial state
$u[0]$ and final target state $u[T]$ both in
 $\textbf{H}(M)$,
% satisfying
% \begin{equation*}
%  E(u[0]), E(u[T])\leq E_M,
% \end{equation*}
there exists some control term $f_0\in C([0, T]; L^2(\mathbb{S}^1))$ compactly supported in $[0, T]\times \omega$ such that
the unique solution of the controlled wave maps equation \eqref{eq:inhomowavemaps} with initial state $u[0]$ and control $f_0$ satisfies $\phi[T]= u[T]$.
\end{thm}

After  this result on the global exact controllability property without any control period restriction, we try to seek global controllability results within optimal control time. Again, starting by working with the one-dimensional target, we  show that $T= 2\pi$ is sufficient  to control the system when the target is $\mathbb{S}^1$. Under this circumstance, the controlled wave maps equations can be transformed into a controlled linear wave equation for which the controllability properties are well-known. This gives the following theorem.
\begin{thm}[Global exact controllability within sharp time for $\mathbb{S}^1$-target] \label{THM-optimaltime-S1}
    The controlled wave maps equation $\phi: \mathbb{R}\times \mathbb{S}^1\rightarrow \mathbb{S}^1$ is globally exactly controllable in   homotopy class. Moreover, this controllability holds in every time $T>T_0$ with
    \begin{equation}\label{Optimal-control-time}
    T_0:= \max_{x\in \mathbb{T}} \left\{\min\{\alpha\geq 0: x+\alpha\in \omega\}\right\}\leq 2\pi.
\end{equation}
and $T_0$ is optimal for this property. \jm{Before it was said that one has this controllability holds in time $T_0$. But it seems to be wrong since we required that the support of the control at any time is included in the open set $\omega$.}
\end{thm}
However, for the more complicated $\mathbb{S}^k$-target case this technique is no longer valid.  This problem becomes much more delicate in the sense that the nonlinear term plays a leading role for large states, especially for states that are far away from geodesics. Indeed, such a sharp control period   problem still remains open for  wave maps equations for general target $\mathcal{M}$. For 1D semilinear wave equations, let us mention \cite{1993-Zuazua-AIHP} by E. Zuazua which gives the optimal time. \\
%It is also noteworthy that   the study on quantitative control and stabilization of  multidimensional (nonlinear) wave equations is limited.   .

This paper is organized as follows. We first present some preliminary results concerning the well-posedness issues of the wave maps equations and the controlled wave maps equations in Section \ref{sec:pre}. Then, Section \ref{sec:stabobs}  is devoted to the proof of Theorem \ref{THM-nounifdecay} concerning generic obstruction on uniform asymptotic stabilization at $2\pi$-energy level. In Section \ref{sec-glo-cont} we show that the controlled wave maps equation is indeed globally exact controllable, namely Theorem \ref{THM-globalexact}. Lastly, some discussion on sharp control time, which is related to Theorem \ref{THM-optimaltime-S1},  is contained in Section \ref{sec-op-contime}.

\section{Some preliminary results}\label{sec:pre}

In this preliminary section we recall some well-posedness and continuous dependence results that will be used later on.
 First of all, thanks to  direct energy estimates, there exists some effectively computable  $Q>1$, depending only on the value of $\|a(x)\|_{L^{\infty}}$,  such that any solution of the damped wave maps equation \eqref{eq:dwm} satisfies
\begin{equation}\label{estimate-with-Q}
   Q^{-1} E(-16\pi)\leq E(0)\leq  E(-16\pi).
\end{equation}
We stress that throughout this paper we work with the energy based regularity, namely $(\phi, \phi_t)(x)\in H^1\times L^2(\Sph^1)$, while the wave maps equation is known to be well-posed in the less regular $H^s$ space with $s>3/4$ \cite{Keel-Tao-1998}.  
 For readers' convenience, let us  start by giving the definition of the inhomogeneous wave maps equations:

\begin{defn}\label{def:sol:open}
    Let $T_1\in \mathbb{R}$, $T_2\in (T_1, +\infty)$.  Let $(g_0, g_1)\in L^2(\mathbb{S}^1; T \mathbb{S}^1)$ and $f\in L^2(T_1, T_2; L^2_x(\mathbb{S}^1; \mathbb{R}^{k+1}))$. A solution to the Cauchy problem of the inhomegeneous wave maps equation 
\begin{equation}\label{eq:cauchyinhomowavemaps}
\begin{cases}
         \Box \phi= \left(|\phi_{t}|^2- |\phi_{ x}|^2\right) \phi+
     f^{\phi^{\perp}}, \; \forall (t, x)\in (T_1, T_2)\times \mathbb{S}^1, \\  (\phi, \phi_t)(T_1)= (g_0, g_1),
     \end{cases}
\end{equation}
is a function $(\phi, \phi_t): [T_1, T_2]\times \mathbb{S}^1\rightarrow T\mathbb{S}^k$ that belongs to $C([T_1, T_2]; H^1\times L^2(\mathbb{S}^1; T\mathbb{S}^k))$
such that, for every $\tau\in (T_1, T_2]$, for every $\varphi(t, x)\in C^1([T_1, \tau]\times \mathbb{S}^1)$ one has 
\begin{align}
    &\;\;\;\;\; \int_{T_1}^{\tau} \int_{\mathbb{S}^1} \left(\phi_x\cdot \varphi_x- \phi_t\cdot \varphi_t+ (|\phi_t|^2- |\phi_x|^2)\phi\cdot \varphi+ f^{\phi^{\perp}}\cdot \varphi\right)(t, x) dt dx \notag\\
    &= \int_{\mathbb{S}^1} \phi_t \cdot \varphi (T_1, x) dx- \int_{\mathbb{S}^1} \phi_t \cdot \varphi (\tau, x) dx. \label{eq:testfunction}
\end{align}
\end{defn}
Similarly, one can define a solution of the damped wave maps equation by replacing $f$ by $a(x) \phi_t$. Concerning the inhomogeneous wave maps equation and the damped wave maps equation, one has the following well-posedness results:

\begin{lem}[Well-posedness of the inhomogeneous wave maps equation and the damped wave maps equation: \cite{Krieger-Xiang-2022}]\label{lem:inhwm}
Let $\bar T >0$.  Denote the domain  $[-\bar T, \bar T]\times \mathbb{S}^1$ by  $D$.

(i) For any initial state  $\phi[0]: \Sph^1\rightarrow T\mathbb{S}^k$ in $H^1_x\times L^2_x$ and any source term $f(t, x): [-\bar T, \bar T]\times \Sph^1\rightarrow \mathbb{R}^{k+1}$ in $L^2_{t,x}(D)$, the inhomogeneous wave maps equation \eqref{eq:inhomowavemaps}
admits a unique solution $\phi[t]$ on the time interval
 $t\in [- \bar T, \bar T]$. This solution takes values in the target $T\Sph^k$:
\begin{gather*}
     (\phi, \phi_t)(t, x)\in  T\mathbb{S}^k.
\end{gather*}
 Moreover, there exists some effectively computable constant $C_w>0$ depending only on the value of $\bar T$ such that  this unique solution verifies the energy estimates\footnote{We use the notation on the standard null coordinates $u = x+t,\,v = t - x$.  Note that under the null coordinates of $(u, v)$ the diamond-shaped region $D$ is not a standard rectangular-shaped region, thus the $L^2_uL^{\infty}_v(D)$-norm ($resp$. $L^2_v L^{\infty}_u(D)$-norm,  $L^{\infty}_u L^2_v$-norm, $L^{\infty}_v L^2_u$-norm and  $L^{2}_uL^2_v$-norm) of a function $f$ can be understood as $\|\tilde f\|_{L^2_uL^{\infty}_v(\tilde D)}$, where we extend $D$ to a rectangle region $\tilde D$ under $(u, v)$-coordinate and extend $f$ to $\tilde f$ trivially as zero in the region $\tilde D\setminus D$.   }
\begin{gather*}
     \|\phi[t]\|_{\dot{H}^1_x\times L^2_x}\leq C_w \left(\|\phi[0]\|_{\dot{H}^1_x\times L^2_x}+  \|f\|_{L^2_{t, x}(D)}\right), \; \forall t\in [-\bar T, \bar T],\\
   \|\phi_v\|_{L^{2}_v L^{\infty}_u\cap L^{\infty}_u L^{2}_v(D)}+  \|\phi_u\|_{L^{2}_u L^{\infty}_v\cap L^{\infty}_v L^{2}_u(D)}\leq C_w \left(\|\phi[0]\|_{\dot{H}^1_x\times L^2_x}+  \|f\|_{L^2_{t, x}(D)}\right).
\end{gather*}

(ii) For any initial state  $\phi[0]: \Sph^1\rightarrow  T\mathbb{S}^k$ in $H^1_x\times L^2_x$, the damped wave maps equation  \eqref{eq:dwm}
admits a unique solution $\phi[t]$ for 
 $t\in \mathbb{R}$ This solution takes values in the target $\Sph^k$:
\begin{gather*}
     (\phi, \phi_t)(t, x)\in  T\mathbb{S}^k.
\end{gather*}
 Moreover,  this unique solution verifies the energy estimates
\begin{gather*}
     \|\phi[t]\|_{\dot{H}^1_x\times L^2_x}\leq C_w \|\phi[0]\|_{\dot{H}^1_x\times L^2_x}, \; \forall t\in [- \bar T, \bar T],\\
   \|\phi_v\|_{L^{2}_v L^{\infty}_u\cap L^{\infty}_u L^{2}_v(D)}+  \|\phi_u\|_{L^{2}_u L^{\infty}_v\cap L^{\infty}_v L^{2}_u(D)}\leq C_w \|\phi[0]\|_{\dot{H}^1_x\times L^2_x}.
\end{gather*}

\end{lem}

Similar to Lemma \ref{lem:inhwm}, and using the same proof, one also has the following well-posedness result for  the free inhomogeneous wave equation.
\begin{lem}[Well-posedness of the inhomogeneous wave equation]\label{lem:freewave}
Let $ \bar T >0$.  There exists an effectively computable  constant $C=C(\bar T)$ such that,  for any $T\in (0, \bar T]$, for any initial state  $\phi[0]: \Sph^1\rightarrow \mathbb{R}^{k+1}\times \mathbb{R}^{k+1}$ in $H^1_x\times L^2_x$ and any source term $f(t, x): [-T,T]\times \Sph^1\rightarrow \mathbb{R}^{k+1}$ in $L^2_{t, x}(D)$, the inhomogeneous wave  equation
\begin{equation}\label{eq:inhomowavee}
    \Box \phi= f \; \textrm{ with } (\phi, \phi_t)(0, x)= \phi[0],
\end{equation}
admits a unique solution $\phi[t]$ in the time interval $t\in [-T, T]$.
 Moreover,  this unique solution verifies the energy estimates
\begin{gather*}
     \|\phi[t]\|_{H^1_x\times L^2_x}\leq C\left(\|\phi[0]\|_{H^1_x\times L^2_x}+ T^{\frac{1}{2}}\|f\|_{L^2_{t,x}([-T, T]\times \Sph^1)}\right), \; \forall t\in [-T, T],\\
   \|\phi_v\|_{L^{2}_v L^{\infty}_u\cap L^{\infty}_u L^{2}_v([-T, T]\times \Sph^1)}+  \|\phi_u\|_{L^{2}_u L^{\infty}_v\cap L^{\infty}_v L^{2}_u([-T, T]\times \Sph^1)}\\
    \leq C \left(\|\phi[0]\|_{\dot{H}^1_x\times L^2_x}+  T^{\frac{1}{2}}\|f\|_{L^2_{t, x}([-T, T]\times \Sph^1)}\right).
\end{gather*}

\end{lem}

\begin{lem}\label{lem-conti-dep-inh} {\rm (Continuous dependence of the inhomogeneous wave maps equation and the damped wave maps equation).}
Let $\bar T >0$, let $M>0$. Denote the region  $[-\bar T, \bar T]\times \mathbb{S}^1$ by $D$.

(i) There exists an effectively computable constant % $a= a(T_0, M)$ and
$C= C(\bar T, M)$ such that,  for any initial states  $\phi[0], \varphi[0]: \Sph^1\rightarrow T\mathbb{S}^k$ in $H^1_x\times L^2_x$ and any source terms $f(t, x), g(t, x): [-\bar T, \bar T]\times \Sph^1\rightarrow \mathbb{R}^{k+1}$ in $L^2_{t, x}(D)$ satisfying
\begin{gather*}
\lVert \phi[0]\lVert_{\dot{H}^1\times L^2}+ \lVert \varphi[0]\lVert_{\dot{H}^1\times L^2}+ \lVert f\lVert_{L^2_{t, x}(D)}+ \lVert g\lVert_{L^2_{t, x}(D)}\leq  M,
%%% \lVert \phi[0]- \varphi[0]\lVert_{H^1\times L^2}+ \lVert f-g\lVert_{L^2_{t, x}(D)}\leq a,
\end{gather*}
 the unique solutions of the inhomogeneous wave maps equations
\begin{equation}\label{eq:phi1}
    \Box \phi= \left(|\phi_{t}|^2- |\phi_{x}|^2\right) \phi+  f^{\phi^{\perp}},  \; (\phi, \phi_t)(0, x)= \phi[0],
\end{equation}
and
\begin{equation}\label{eq:varphi1}
    \Box \varphi= \left(|\varphi_{t}|^2- |\varphi_{x}|^2\right) \varphi+  g^{\varphi^{\perp}},  \; (\varphi, \varphi_t)(0, x)= \varphi[0],
\end{equation}
satisfy  the following energy estimates, where $w:= \phi- \varphi$,
\begin{gather*}
   \|(w_x, w_t, w)\|_{L^{\infty}_t L^2_x(D)}+  \|w_u\|_{L^{2}_u L^{\infty}_v\cap L^{\infty}_v L^{2}_u(D)}+  \|w_v\|_{L^{2}_v L^{\infty}_u\cap L^{\infty}_u L^{2}_v(D)} \\
   \leq C\left(\lVert w[0]\lVert_{H^1\times L^2}+ \lVert f-g\lVert_{L^2_{t, x}(D)}\right).
\end{gather*}

(ii) For any initial states  $\phi[0], \varphi[0]: \Sph^1\rightarrow T\mathbb{S}^k$ in $H^1_x\times L^2_x$  satisfying
\begin{gather*}
\lVert \phi[0]\lVert_{\dot{H}^1\times L^2}+ \lVert \varphi[0]\lVert_{\dot{H}^1\times L^2}\leq  M,
\end{gather*}
 the unique solutions of  the damped wave maps equations
\begin{equation}
    \Box \phi= \left(|\phi_{t}|^2- |\phi_{x}|^2\right) \phi+  a(x) \phi_t,  \; (\phi, \phi_t)(0, x)= \phi[0],  \notag
\end{equation}
and
\begin{equation}
    \Box \varphi= \left(|\varphi_{t}|^2- |\varphi_{x}|^2\right) \varphi+   a(x) \varphi_t,  \; (\varphi, \varphi_t)(0, x)= \varphi[0], \notag
\end{equation}
satisfy  the following energy estimates, where $w:= \phi- \varphi$,
\begin{gather*}
   \|(w_x, w_t, w)\|_{L^{\infty}_t L^2_x(D)}+  \|w_u\|_{L^{2}_u L^{\infty}_v\cap L^{\infty}_v L^{2}_u(D)}+  \|w_v\|_{L^{2}_v L^{\infty}_u\cap L^{\infty}_u L^{2}_v(D)} \\
   \leq C \lVert w[0]\lVert_{H^1\times L^2}.
\end{gather*}
\end{lem}
\begin{proof}[Proof of Lemma \ref{lem-conti-dep-inh}]
We only prove the result concerning the inhomogeneous equation, as the same proof yields the continuous dependence result on the damped wave maps equation.   The proof is straightforward and is based on a bootstrap argument. Let $T\in (0, \bar T)$ with its value to be fixed later on.  For ease of notation, we shall denote the region $(0, T)\times \mathbb{S}^1$ as $Q_T$, and define the norm $\mathcal{W}_T$ by
\begin{equation}
\lVert w\lVert_{\mathcal{W}_T}:= \|(w_x, w_t, w)\|_{L^{\infty}_t L^2_x(Q_T)}+  \|w_u\|_{L^{2}_u L^{\infty}_v\cap L^{\infty}_v L^{2}_u(Q_T)}+  \|w_v\|_{L^{2}_v L^{\infty}_u\cap L^{\infty}_u L^{2}_v(Q_T)}.
\end{equation}
Thanks  to Lemma \ref{lem:inhwm},  we know that  $\lVert \phi\lVert_{\mathcal{W}_{T}}$ and $\lVert \varphi\lVert_{\mathcal{W}_{T}}$ are uniformly bounded by some constant depending on $\bar T$ and $M$.

One immediately deduces from Equations \eqref{eq:phi1} and \eqref{eq:varphi1} that
\begin{equation}
\Box w= (w_{u}\cdot \phi_v) \phi+ (w_{v}\cdot \varphi_u) \phi+ (\varphi_{u}\cdot \varphi_v) w+ f^{\phi^{\perp}}-  g^{\varphi^{\perp}}=: F,   \notag
\end{equation}
thus,   according to   Lemma \ref{lem:freewave} for estimates on $w$, one obtains
\begin{equation}
\lVert w\lVert_{\mathcal{W}_{T}}\lesssim\footnote{Throughout this paper we use the notation $a\lesssim b$ to indicate that $a\leq C b$ where $C$ is some effectively computable constant.} \lVert w[0]\lVert_{H^1\times L^2}+ T^{\frac{1}{2}}\lVert F\lVert_{L^2_{t, x}(Q_T)}.   \notag
\end{equation}

Next, we estimate the value of $F$ in order to close the loop. Successively there is
\begin{align*}
\lVert(w_{u}\cdot \phi_v) \phi\lVert_{L^2_{t, x}(Q_T)}&\lesssim  \lVert \phi_v\lVert_{L^2_v L^{\infty}_u(Q_T)}  \lVert w_u \lVert_{L^{\infty}_v L^2_u(Q_T)}\lesssim \lVert w\lVert_{\mathcal{W}_{T}}  \\
\lVert (w_{v}\cdot \varphi_u) \phi\lVert_{L^2_{t, x}(Q_T)}&\lesssim \lVert w\lVert_{\mathcal{W}_{T}}  \\
\lVert (\varphi_{u}\cdot \varphi_v) w\lVert_{L^2_{t, x}(Q_T)}&\lesssim \lVert w\lVert_{\mathcal{W}_{T}},
\end{align*}
and
\begin{align*}
f^{\phi^{\perp}}-  g^{\varphi^{\perp}}
=& f- \langle f, \phi\rangle \phi+ g- \langle g, \varphi\rangle \varphi \\
=& (f- g)+ \langle f, \phi\rangle w+ \langle f, w\rangle \varphi+ \langle f-g, \varphi\rangle \varphi
\end{align*}
satisfies
\begin{equation*}
\lVert f^{\phi^{\perp}}-  g^{\varphi^{\perp}}\lVert_{L^2_{t, x}(Q_T)}\lesssim \lVert f-g \lVert_{L^2_{t, x}(Q_T)}+ \lVert w\lVert_{\mathcal{W}_{T}}.
\end{equation*}
Therefore, by choosing $T$ sufficiently small   depending only on the values of $\bar T$ and $M$, one has
\begin{equation}
\lVert w\lVert_{\mathcal{W}_{T}}\lesssim \lVert w[0]\lVert_{H^1\times L^2}+ \lVert f-g \lVert_{L^2_{t, x}(Q_T)}.   \notag
\end{equation}
This ends the proof of the first property (i).
\end{proof}

Let us be given a time-varying feedback law $F$, namely a map:
\begin{align}
    F: \mathbb{R}\times H^1_x\times L^2_x(\mathbb{S}^1; T\mathbb{S}^k)&\rightarrow L^2_x(\mathbb{S}^1; \mathbb{R}^{k+1}), \label{maps:F}\\
    (t; (g_0, g_1)(\cdot))&\mapsto F(t, (g_0, g_1)(\cdot))(x)|_{x\in \mathbb{S}^1}. \notag
\end{align}
We assume that this map is a Carathéodory map in the following sense: 
\begin{itemize}
    \item[($\mathcal{P}1$)] $\forall R>0, \exists C_B(R)>0$ such that $\left( \|(g_0, g_1)\|_{\dot H^1\times L^2}\leq R \Rightarrow \|F(t, g_0, g_1)\|_{L^2}\leq C_B(R)\right)$. Moreover, we assume that the function $C_B(R)$ is nondecreasing. 
   \item[($\mathcal{P}2$)] $\forall (g_0, g_1)\in H^1_x\times L^2_x(\mathbb{S}^1; T\mathbb{S}^k)$, the function $t\in \mathbb{R}\mapsto F(t, g_0, g_1)\in L^2(\mathbb{S}^1; \mathbb{R}^{k+1})$ belongs to $L^2_{\textrm{loc}}(\mathbb{R}; L^2(\mathbb{S}^1; \mathbb{R}^{k+1}))$
   \item[($\mathcal{P}3$)] for almost every $t\in \mathbb{R}$, the function $(g_0, g_1)\in H^1_x\times L^2_x(\mathbb{S}^1; T\mathbb{S}^k)\mapsto F(t, g_0, g_1)\in L^2(\mathbb{S}^1; \mathbb{R}^{k+1})$ is continuous. 
\end{itemize}
We also assume that $F$ is a Lipschitz map in the sense that 
\begin{itemize}
    \item[($\mathcal{P}4$)]  for every $R>0$, there exists $K(R)>0$ such that 
\begin{gather*}
    \left(\|(\phi_1, \phi_{1t})\|_{\dot H^1\times L^2}\leq R, \|(\phi_2, \phi_{2t})\|_{\dot H^1\times L^2}\leq R \right) \Rightarrow \\
    \left( \|F(t, \phi_1, \phi_{1t})- F(t, \phi_2, \phi_{2t})\|_{L^2}\leq K(R) \|(\phi_1- \phi_2, \phi_{1t}-\phi_{2t})\|_{H^1\times L^2}\right). 
\end{gather*}
\end{itemize}

We are interested in the stability of the  closed-loop system \eqref{closed-loop-timevarying} with $F$  as defined in \eqref{maps:F}.  First, we give the definition of solutions to this system:
\begin{defn}\label{def:sol:clos}
      Let us be given a map $F$ in form of \eqref{maps:F} that satisfies conditions ($\mathcal{P}1$)--($\mathcal{P}3$).  Let $T_1\in \mathbb{R}$.  Let $(g_0, g_1)\in L^2(\mathbb{S}^1; T \mathbb{S}^1)$. A function $(\phi, \phi_t)$  is a solution to  the Cauchy problem 
\begin{equation}\label{eq:cauchyclosedloop}
\begin{cases}
         \Box \phi(t, x)= \left(|\phi_{t}|^2- |\phi_{ x}|^2\right)(t, x) \phi(t, x)+
     F(t, \phi(t, \cdot), \phi_t(t, \cdot))(x)^{\phi(t, x)^{\perp}}, \; \forall (t, x)\in (T_1, +\infty)\times \mathbb{S}^1, \\  (\phi, \phi_t)(T_1, x)= (g_0, g_1)(x),
     \end{cases}
\end{equation}
if there exists an interval $I$ with a non-empty interior satisfying $I \cap (-\infty, T_1]= \{T_1\}$ such that 
 $(\phi, \phi_t): I\times \mathbb{S}^1\rightarrow T\mathbb{S}^k$ that belongs to $C([T_1, T_2]; H^1\times L^2(\mathbb{S}^1; T\mathbb{S}^k))$  is a solution to the Cauchy problem \eqref{eq:cauchyinhomowavemaps} with $f(t, x)= F(t, \phi(t, \cdot), \phi_t(t, \cdot))(x)$. The interval $I$ is denoted by $D(\phi, \phi_t)$.

We say that a solution $(\phi, \phi_t)$ is maximal if, for every solution $(\tilde \phi, \tilde \phi_t)$ to the Cauchy problem \eqref{eq:cauchyclosedloop} such that 
\begin{gather*}
    D(\phi, \phi_t)\subset D(\tilde \phi, \tilde \phi_t), \\
    (\phi, \phi_t)(t, \cdot)=  (\tilde \phi, \tilde \phi_t)(t, \cdot) \textrm{ for every } t\in D(\phi, \phi_t),
\end{gather*}
one has 
\begin{equation*}
    D(\phi, \phi_t)= D(\tilde \phi, \tilde \phi_t).
\end{equation*}
\end{defn}

\begin{defn}
   Let us be given a map $F$ in form of \eqref{maps:F} that satisfies conditions ($\mathcal{P}1$)--($\mathcal{P}3$). Let $I$ be a nonempty interval of $\mathbb{R}$. A function $(\phi, \phi_t)$ is a solution of the closed-loop system \eqref{closed-loop-timevarying} on $I$ if, for every $[T_1, T_2]\subset I$, the restriction of $(\phi, \phi_t)$ to $[T_1, T_2]\times \mathbb{S}^1$ is a solution of the Cauchy problem \eqref{eq:cauchyclosedloop} with initial state $(\phi, \phi_t)(T_1, \cdot)$.
\end{defn}

For such a closed-loop system we have the following global well-posedness  and continuous dependence result,  the proof of which we put in Appendix \ref{Sec:App:A}.

\begin{prop}\label{lem:wellclosegene}
Let us be given a map $F$ in form of \eqref{maps:F} that satisfies conditions ($\mathcal{P}1$)--($\mathcal{P}4$). 

(i). For every $R\in (0, +\infty)$, there exists a time $T(R)>0$ and a constant $L(R)>0$ such that, 
\begin{itemize}
    \item for every $T_1\in \mathbb{R}$  and for every initial state $(g_0, g_1)$ satisfying $\|(g_0, g_1)\|_{\dot H^1\times L^2}\leq R$, the Cauchy problem \eqref{eq:cauchyclosedloop}   has one and only one solution on $[T_1, T_1+ T(R)]$. 
    \item  for every $T_1\in \mathbb{R}$, for every $(g_0, g_1)$ (\textit{resp}. $(\tilde g_0, \tilde g_1)$) satisfying $\|(g_0, g_1)\|_{\dot H^1\times L^2}\leq R$ (\textit{resp}.$\|(\tilde g_0, \tilde g_1)\|_{\dot H^1\times L^2}\leq R$), the unique solution of the Cauchy problem \eqref{eq:cauchyclosedloop} with initial state $(g_0, g_1)$(\textit{resp}. $(\tilde g_0, \tilde g_1)$) on $[T_1, T_1+ T(R)]$, $(\phi, \phi_t)$ (\textit{resp}. $(\tilde \phi, \tilde \phi_t)$) satisfies 
    \begin{equation*}
        \|(\phi, \phi_t)- (\tilde \phi, \tilde \phi_t)\|_{H^1\times L^2}(t)\leq L(R) \|(g_0, g_1)- (\tilde g_0, \tilde g_1)\|_{H^1\times L^2}, \forall t\in [T_1, T_1+ T(R)].
    \end{equation*}
\end{itemize}

(ii). For every $T_1\in \mathbb{R}$, for every initial state $(g_0, g_1)\in H^1\times L^2(\mathbb{S}^1; T\mathbb{S}^k)$, the Cauchy problem   has one and only one maximal solution $(\phi, \phi_t)$. If $D(\phi, \phi_t)$ is not equal to $[T_1, +\infty)$, then there exists some $\tau\in \mathbb{R}$ such that $D(\phi, \phi_t)= [T_1, \tau)$ and one has 
\begin{equation*}
    \lim_{t\rightarrow \tau^-}\|(\phi, \phi_t)\|_{H^1\times L^2}= +\infty.
\end{equation*}
\end{prop}

% \begin{prop}\label{prop:inregwellpo}
%     Let us be given a map $F$ in form of \eqref{maps:F} that satisfies conditions ($\mathcal{P}1$)--($\mathcal{P}3$), and that, the function $C_B(R)|_{R>0}$ satisfies 
%     \begin{equation*}
%         \int_0^{+\infty} \frac{1}{\max \{1, C_B(R)\}} dR= +\infty,
%     \end{equation*}
% then, for every $T_1\in \mathbb{R}$ and for every initial state $(g_0, g_1)\in H^1\times L^2(\mathbb{S}^1; \mathbb{S}^k)$, the Cauchy problem \eqref{eq:cauchyclosedloop} has at least one maximal solution $(\phi, \phi_t)$ such that $D(\phi, \phi_t)= [T_1, +\infty)$.
% \end{prop}

\section{$2\pi$-energy threshold for uniform asymptotic stabilization}\label{sec:stabobs}
As stated in the introduction, in this section we explore global stabilization of the closed-loop system \eqref{closed-loop-timevarying}. Instead of the classical damping feedback one is allowed to use any other  time-varying continuous feedback law \eqref{feedbacklaw-def}. An intuitive idea is to construct a feedback law such that a given harmonic map, denoted by $(\mathcal{Q}, 0)$, is no longer a stationary state and hope that the energy dissipates around this state. This may bypass the obstruction caused by $(\mathcal{Q}, 0)$.

\subsection{Damping and harmonic maps. }

Damping stabilization, a standard tool for stabilization, has been extensively studied in the literature for both linear and nonlinear dispersive equations including the wave equations, Schr\"{o}dinger equations, KdV among others. In this special circumstance, the feedback law is chosen as
\begin{equation*}
F(t; \phi[t]):= a(x) \phi_t(t, x) \; \textrm{ with } a(x)\geq 0, \; \textrm{ supp }a \subset  \omega,\; a \not =0.
\end{equation*}
Typically,  it relies on the so-called compactness-uniqueness method and provides nearly optimal theoretical results, while the drawback of this method is usually the lack of a quantitative description. Recently, the first two authors have studied the quantitative control of some basic dispersive models, namely KdV \cite{Krieger-Xiang-kdv} and the radial focusing Klein--Gordon equation \cite{krieger2020boundary}, and have provided a more quantitative description. Another important property of damping stabilization is that it sometimes leads to global stabilization results, notably for defocusing dispersive equations like nonlinear wave equations \cite{Laurent-2011} among others, in contrast to other stabilization techniques usually limited to local results such as backstepping \cite{Gagnon-Hayat-Xiang-Zhang}  and spectral type methods \cite{MR1745475, Xiang-heat-2020}.   In \cite{Krieger-Xiang-2022}, the last two authors found that the damping stabilization technique leads to  semi-global exponential  stabilization of the wave maps equation. As explained in the introduction, the energy restriction comes from harmonic maps and cannot be improved.
%where harmonic maps (non-trivial stationary state for both wave maps and damped wave maps) appear as an limitation for damping stabilization.
% \begin{itemize}
%     \item Can we get global stabilization using other feedback laws instead of damping, for which harmonic maps are no longer stationary states?
% \end{itemize}

\subsection{Topology and obstruction to stabilization.}\label{sec-highlevel}

\begin{proof}[Proof of Theorem~\ref{THM-nounifdecay}]
We argue by contradiction. We assume the existence of a
 time-varying feedback law (which is not necessarily linear) $F$ satisfying conditions $(\mathcal{P}1)-(\mathcal{P}4)$,
    \begin{equation*}
      F: \mathbb{R}\times H^1_xL^2_x(\mathbb{S}^1;T\mathbb{S}^k)\rightarrow  L^2_x(\mathbb{S}^1),
    \end{equation*}
such that     the closed-loop system \eqref{closed-loop-timevarying} is uniformly asymptotically stable in
    $\textbf{H}(2\pi)$. Let us define the flow for  solutions of the closed-loop wave maps system with initial time equal to 0:
\begin{equation*}
\begin{array}{ccc}
    \Phi= (\Phi_1, \Phi_2): \R \times H^1_xL^2_x(\Sph^1;T\Sph^k)&\rightarrow& H^1_xL^2_x(\Sph^1;T\Sph^k)\\
    (t, u[0])&\mapsto& \phi[t],
\end{array}
\end{equation*}
where $\phi[t]$ is the unique solution of
\begin{gather*}
 \Box \phi= \left(|\phi_{t}|^2- |\phi_{ x}|^2\right) \phi+
     F(t; \phi[t])^{\phi^{\perp}}, \;  (\phi, \phi_t)(0)= u[0].
     \end{gather*}
Indeed, the existence of such a unique solution is given by Proposition \ref{lem:wellclosegene}. 
By Definition~\ref{DEF-usstab} and the asymptotic stability assumption, there exists a $\mathcal{KL}$ function $h$  such that
\begin{multline}\label{estEphit-2pi}
    E(\Phi(t,(\phi,0)))\leq h\left(\int_{\mathbb{S}^1} |\phi_x|^2 dx,t\right),
     \\  \forall t\in (0, +\infty),\;   \forall \phi \in H^1(\mathbb{S}^1;\mathbb{S}^k) \text{ such that } \int_{\mathbb{S}^1} |\phi_x|^2 dx\leq 2\pi.
\end{multline}
In particular, for every $\delta>0$ there exists $T= T(\delta)>0$ such that $\forall \phi \in H^1(\mathbb{S}^1;\mathbb{S}^k)$  \text{ satisfying } 
\begin{equation*}
    \int_{\mathbb{S}^1} |\phi_x|^2 dx\leq 2\pi,
\end{equation*}
 there is
\begin{equation}
\label{phi1<2}
   |\Phi_1(T, (\phi, 0))(x_1)- \Phi_1(T, (\phi, 0))(x_2)|<\delta,\;  \forall x_1, x_2\in \mathbb{S}^1.
\end{equation}

\noindent $\bullet$ {\bf Case $k=1$.}   Let $\phi \in H^1(\mathbb{S}^1;\mathbb{S}^1)$. Observe that, for any given time $t$,  $\Phi_1(t, (\phi, 0))$ is in $C^0(\mathbb{S}^1;\mathbb{S}^1)$ and hence has a degree. Since $t\mapsto \Phi_1(t,(\phi, 0))\in C^0(\mathbb{S}^1;\mathbb{S}^1)$ is continuous, this degree does not depend on time. Note that \eqref{phi1<2} implies that (see the case $k=2$ for more details)
\begin{equation}
\label{k=1degree0}
    \deg(x\in \Sph^1\mapsto \Phi_1(T(\delta=1),(\phi, 0))(x) \in \Sph^1)= 0.
\end{equation}
We take
\begin{equation}
\label{defphix=x}
\phi(x):=x.
\end{equation}
Then
\begin{equation*}
 \int_{\mathbb{S}^1}|\phi_x|^2 dx= 2\pi \text{ and }
\deg(x\in \Sph^1\mapsto \Phi_1(0,(\phi, 0))(x) \in \Sph^1)= \deg(\phi)= 1,
\end{equation*}
which leads to a contradiction due to \eqref{k=1degree0} and concludes the proof of Theorem~\ref{THM-nounifdecay} in the case $k=1$.
\begin{remark}
The above proof to the obstruction of asymptotic stabilizability comes in fact from the following obstruction to controllability: it is not possible to move from $(\phi,0)$, with $\phi$ defined in \eqref{defphix=x},  to $(u_0,u_1)\in H^1_xL^2_x(\Sph^1;T\Sph^1)$ if the $\deg(u_0)\not =1$.
\end{remark}

\noindent $\bullet$ {\bf Case $k= 2$.} In this case the system is globally controllable thanks to Theorem \ref{THM-globalexact}.
Thus the topological properties of $\mathbb{S}^2$ cannot provide any obstruction to controllability. However, the topological aspect for asymptotic stabilization is  different from the one of controllability: there is a topological property of $\Sph^2$ (namely the non triviality of the homology group $H^2(\Sph^2)$) which leads to an obstruction to uniform asymptotic stabilization as we are going to see now.

In \cite{Krieger-Xiang-2022} it is shown that $2\pi$ is an energy threshold for damping stabilization. Let us show that $2\pi$ is exactly the energy threshold for uniform asymptotic stabilization whatever the feedback used. Here we have chosen to work with the $\Sph^2$-target to simplify the notations, leaving the situation of higher dimension targets $\Sph^k$ to the next case.  Let $A: \mathbb{S}^1_s \times \mathbb{S}^1_x \rightarrow \mathbb{S}^2$ be defined by
\begin{equation}\label{def_gs2}
    A(s, x):=
    \begin{cases}
        (\sin{s} \cos{x}, \sin{s} \sin{x}, \cos{s})^T, \; \forall s\in [0, \pi],
\\
(-\sin{s} \cos{x}, \sin{s} \sin{x}, \cos{s})^T, \; \forall s\in (\pi, 2\pi).
    \end{cases}
\end{equation}
In \eqref{def_gs2} and until the end of Section~\ref{sec-highlevel} we identify $\Sph^1$ with $\R/2\pi\mathbb{Z}$.   This $A$, as one can easily check, leads to a
 $\gamma \in C^0(\Sph^1_s; C^{1}(\mathbb{S}^1_x; \mathbb{S}^2))$ if one requires that
\begin{equation}
\gamma(s)(x):=A(s,x).
\end{equation}
A key observation is that the degree of $A$ is not $0$. More precisely, one has the following lemma.
\begin{lem}\label{lemma-degreeA=2}
The degree of $A$ is $2$.
\end{lem}
\begin{proof}
Observe that $p:= (1, 0, 0)$ is a regular point of the map $A$ such that
\begin{equation*}
A^{-1}(p)= \{(\pi/2, 0), (3\pi/2, 0)\}.
\end{equation*}

Around the point  $(\pi/2, 0)$, more precisely for the region $s\in (0, \pi)$, the Jacobi matrix of A is given by
\begin{equation*}
    J_A:= (\partial_s A, \partial_x A, A)= \begin{pmatrix}
    \cos{s} \cos{x}& -\sin{s} \sin{x}  & \sin{s} \cos{x} \\
    \cos{s} \sin{x}& \sin{s} \cos{x} & \sin{s} \sin{x}\\
    -\sin{s} & 0 & \cos{s}
    \end{pmatrix},
\end{equation*}
which leads to
\begin{equation*}
    \text{det} J_A= \sin{s} >0.
\end{equation*}

\noindent Around the point  $(3\pi/2, 0)$, more precisely for the region $s\in (\pi, 2\pi)$, the Jacobi matrix of A is given by
\begin{equation*}
    J_A:= (\partial_s A, \partial_x A, A)= \begin{pmatrix}
    -\cos{s} \cos{x}& \sin{s} \sin{x}  & -\sin{s} \cos{x} \\
    \cos{s} \sin{x}& \sin{s} \cos{x} & \sin{s} \sin{x}\\
    -\sin{s} & 0 & \cos{s}
    \end{pmatrix},
\end{equation*}
which leads to
\begin{equation*}
    \text{det} J_A= -\sin{s} >0.
\end{equation*}

\noindent Therefore, the degree of the map $A$ is
\begin{equation*}
    \deg(A)= 1+1= 2.
\end{equation*}

\end{proof}

Note that
\begin{equation*}
    E((\gamma(s),0))= \int_{\Sph^1} (\sin{s})^2 dx= 2\pi (\sin{s})^2\leq 2\pi, \; \forall s \in \Sph^1.
\end{equation*}
Moreover, thanks to \eqref{phi1<2}, by taking $T= T(\delta= 2)$ one has
\begin{equation}\label{Phi1closetoa}
    |\Phi_1(T,(\gamma(s), 0))(x)- a(s)|<2,\; \forall s \in \Sph^1,\; \forall x \in \Sph^1,
\end{equation}
with
\begin{equation}
\label{def-a}
a(s):= \Phi_1(T,(\gamma(s), 0))(0)\in \Sph^2.
\end{equation}
Note that, by Lemma~\ref{lem-conti-dep-inh}  and \eqref{def-a}, $a\in C^0(\mathbb{S}^1;\mathbb{S}^2)$. 
Indeed, thanks to Proposition \ref{lem:wellclosegene}, there exists some $L>0$ such that for every $s_1, s_2$ belong to $\mathbb{S}^1$, and for every $t\in [0, T]$, we have 
\begin{align*}
    \|\Phi_1(t, (\gamma(s_1), 0))(\cdot)- \Phi_1(t, (\gamma(s_2), 0))(\cdot)\|_{H^1}&\leq L \|(\gamma(s_1), 0)(\cdot)-  (\gamma(s_2), 0)(\cdot)\|_{H^1\times L^2}, \\
    &\leq C L \|\gamma(s_1)(\cdot)-  \gamma(s_2)(\cdot)\|_{C^1}.
\end{align*}
This implies that 
\begin{equation*}
    \Phi_1(t, (\gamma(s), 0))(x)\in C^0\left(S^1_{s}; C^0([0, T]; H^1(\mathbb{S}^1_x; \mathbb{S}^2))\right)\hookrightarrow C^0\left(S^1_{s}; C^0([0, T]; C^0(\mathbb{S}^1_x; \mathbb{S}^2))\right)
\end{equation*}

Now we are able to define a continuous map
\begin{equation}
\begin{array}{ccc}
    \mathcal{H}_1: [0, T]\times \Sph^1_s\times \Sph^1_x&\rightarrow& \Sph^2
\\
    (t; s, x)&\mapsto& \Phi_1(t, (\gamma(s), 0))(x)
\end{array}
\end{equation}
Extend $ \mathcal{H}_1$ to $t\in [T, T+1]$ as
\begin{equation}\label{defH1}
     \mathcal{H}_1(T+ r; s, x):= \frac{(1-r)  \mathcal{H}_1(T; s, x)+ r a(s)}{|(1-r)  \mathcal{H}_1(T; s, x)+ r a(s)|}\in \mathbb{S}^2.
\end{equation}
Note that the denominator appearing in \eqref{defH1} never vanishes as one can see by using \eqref{Phi1closetoa} and the fact that both $\mathcal{H}_1(T; s, x)$ and $a(s)$ are in $\mathbb{S}^2$.
Clearly, $\mathcal{H}_1: [0, T+1]\times \Sph^1_s\times \Sph^1_x\rightarrow \Sph^2$ is a continuous map which shows
that  $ \mathcal{H}_1(t=0): \Sph^1_s\times \Sph^2_x\rightarrow \Sph^2$ is homotopic to
$ \mathcal{H}_1(t=T+1): \Sph^1_s\times \Sph^1_x\rightarrow \Sph^2$. In particular
\begin{equation}
\label{degH0T}
\deg ( (s,x)\in  \Sph^1_s\times \Sph^1_x\mapsto \mathcal{H}_1(0;s,x) \in \Sph^2)= \deg ( (s,x)\in  \Sph^1_s\times \Sph^1_x\mapsto \mathcal{H}_1(T+1;s,x) \in \Sph^2).
\end{equation}

On the one hand, by Lemma~\ref{lemma-degreeA=2}, we know that
\begin{equation}
\label{degH0A}
\deg ( (s,x)\in  \Sph^1_s\times \Sph^1_x\mapsto \mathcal{H}_1(0;s,x) \in \Sph^2)= \deg A = 2.
\end{equation}
On the other hand,
\begin{equation}
     \mathcal{H}_1(t= T+1)= B: (s, x)\mapsto a(s).     \notag
\end{equation}
Since the value of $B(s, x)= a(s)$ only depends on one variable, its degree is 0.
This is in contradiction with the fact that, by \eqref{degH0T} and  \eqref{degH0A},
\begin{equation*}
    2=\deg A= \deg B.
\end{equation*}
This concludes the proof of Theorem~\ref{THM-nounifdecay} in the case $k=2$.

\noindent $\bullet$ {\bf Case $k\geq 3$.} Let us start with the case $k=3$. One defines $A_2:\Sph^1\times\Sph^1\times\Sph^1\rightarrow \Sph^3$  and $\gamma_2\in C^0(\Sph^1\times\Sph^1 ;H^1(\Sph^1;\Sph^3))$ by
\begin{equation}\label{def_gs3}
    A_2(s_1,s_2, x):=
    \begin{cases}
        (\sin{s_1}A(s_2,x)^T , \cos{s_1})^T, \; \forall s_1\in [0, \pi],
\\
(-\sin{s_1} A(s_2,x)^T, \cos{s_1})^T, \; \forall s_1\in (\pi, 2\pi),
    \end{cases}
\end{equation}
and
\begin{equation}
\gamma_2(s_1,s_2)(x):=A_2(s_1,s_2,x).
\end{equation}
We have
\begin{gather*}
    E((\gamma_2(s_1,s_2),0))= 2\pi\sin{s_1}^2\sin{s_2}^2 \leq 2\pi,
     \deg(A_2)= 4\not = 0.
\end{gather*}
As in the case $k=2$, this gives the proof of Theorem~\ref{THM-nounifdecay} for $k=3$. For $k>3$ one defines by induction on $k$ $A_k:(\Sph^1)^{k+1}\rightarrow \Sph^{k+1}$
\begin{equation}\label{def_gsk}
    A_{k}(s_1,s_2,\ldots,s_k, x):=
    \begin{cases}
        (\sin{s_1}A_{k-1}(s_2,\ldots,s_k,x)^T , \cos{s_1})^T, \; \forall s_1\in [0, \pi],
\\
(-\sin{s_1} A_{k-1}(s_2,\ldots,s_k,x)^T, \cos{s_1})^T, \; \forall s_1\in (\pi, 2\pi).
    \end{cases}
\end{equation}
and define and $\gamma_k\in C^0((\Sph^1)^k ; H^1(\Sph^1;\Sph^{k+1}))$ by
\begin{equation}
\gamma(s)(x):=A(s,x).
\end{equation}
We have
\begin{gather*}
    E((\gamma_k(s_1,\ldots,s_k),0))= 2\pi\sin{s_1}^2\ldots \sin{s_k}^2 \leq 2\pi
    \text{ and }\deg(A_k)= 2^k\not = 0,
\end{gather*}
which, again, leads to a contradiction.

\end{proof}

\begin{remark}
\label{asto-a-given-constant}
We may  consider the stronger ``state uniform asymptotic stabilization" property for which   the solution converges to some given stationary state $(p, 0)$, with $p\in \Sph^k$. Then \eqref{estEphit} is  replaced by
\begin{equation}\label{decay-phi}
    \|\phi[t]- (p, 0)\|_{H^1_x\times L^2_x}\leq h(\|\phi[0]- (p, 0)\|_{H^1_x\times L^2_x},t), \; \forall t\in (0, +\infty), \, \forall \phi[0] \in \textbf{H}(e).
\end{equation}

Obviously, this type of  stabilization  is stronger than the former one given in Definition \ref{DEF-usstab}.
%One shall remark that this difference comes from the fact that the geometric wave equations admit infinitely many trivial stationary states.
If the energy is sufficiently small, then there exists some point $p\in \mathbb{S}^k$ such that the state is close to $(p, 0)$ in $H^1_x\times L^2_x$.  The difference  between these two types of stabilization comes  from  the fact that, for the former stabilization,  solutions with different initial states can converge to different  points. It turns out that there is an obstruction to state uniform asymptotic stabilization already with $e=0$. It suffices for that to consider $\gamma :\Sph^k_s\rightarrow H^1(\Sph^1_x;\Sph^k)$ defined by
\begin{equation}
\gamma(s)(x)=s, \forall s\in \Sph^k, \forall x\in \Sph^1.
\end{equation}
Then
\begin{gather}
    E((\gamma(s),0))= 0,
\\
\label{degree1}\deg(s\in \Sph^k\mapsto\gamma(s)(0)\in \Sph^k)= 1.
\end{gather}
Moreover, from \eqref{decay-phi} one has, for $T>0$ large enough,
\begin{gather}
\label{degree0}
   \deg(s\in \Sph^k\mapsto \Phi_1(T,(\gamma(s),0))(0)\in \Sph^k)= 0.
\end{gather}
However, since 
\begin{align*}
    [0, T]\times \Sph^k_s&\rightarrow \Sph^k,\\
   (t, s)&\mapsto \Phi_1(t, (\gamma(s), 0))(0)
\end{align*}
is continuous, there is
\begin{multline}
  \deg(s\in \Sph^k\mapsto\gamma(s)(0)\in \Sph^k)=\deg(s\in \Sph^k\mapsto \Phi_1(0,(\gamma(s),0))(0)\in \Sph^k)
   \\ =\deg(s\in \Sph^k\mapsto \Phi_1(T,(\gamma(s),0))(0)\in \Sph^k),
\end{multline}
which leads to a contradiction with \eqref{degree1} and \eqref{degree0}.
\end{remark}

\section{Global  controllability of the wave maps equation for $\Sph^k$-target, $k\geq 2$}\label{sec-glo-cont}

% \begin{thm}\label{THM-localexact}
% There exist some effectively computable $T_{l}>0$ and $C_l>0$, such that for any initial state  $u[0]\in \Sph^k\times T\Sph^k$ satisfying
% \begin{equation*}
%  E(0)\leq \pi,
% \end{equation*}
% there exists some control term $f_0\in C([0, T_l]; L^2(\mathbb{S}^1))$ compactly supported in $[0, T]\times \omega$ such that
% the unique solution of the controlled wave maps equation with initial state $u[0]$ and control $f_0$ satisfies $\phi[T_0]= (0, 0)$.
% \end{thm}

\subsection{The three-step strategy for controllability to constant states}
By constant states we mean states $(\phi_1,\phi_2)\in H^1_xL^2_x(\Sph^1;T\Sph^k)$ such that $\phi_2=0$ and there exists $p\in \Sph^k$ such that $\phi_1(x)=p$. Due to the time-reversal property of the inhomogeneous wave maps equation it suffices to establish the {\it exact controllability to a given constant state} to prove Theorem \ref{THM-globalexact}, namely, for any given state and given constant state, prove the existence of a control that steers in finite time the control system from this given state to the given constant state. For this purpose we perform a three-step strategy:
\begin{itemize}
    \item  Firstly, we show that the damping effect dissipates the system's energy such that the unique solution converges asymptotically to harmonic maps, namely its energy must converge to some critical value $2\pi N^2$. This generalizes Proposition \ref{thm:semicontrol} $(ii)$ for which the energy is strictly limited from above by $2\pi$.  See Theorem \ref{thm-converharm} for details.
    \item Secondly, we continue to decrease the system's energy by designing explicit controls  such that it becomes strictly smaller than the preceding critical value $2\pi N^2$. Subsequently, we again perform the damping technique to make the energy converge to $2\pi N_1^2$ with $N_1< N$, and design precise controls to decrease it below this value.  We iterate these two steps until the energy of the solution is smaller than $2\pi$.  See Theorem \ref{thm-decreseenergy} for details.
    \item Lastly, we use the low-energy exact controllability result Proposition \ref{thm:semicontrol} $(ii)$ to conclude the proof of exact controllability to the given constant state.
\end{itemize}
To be more precise,  the following two theorems exactly correspond to the above-illustrated strategy.  To better formulate the results,  we  introduce the so-called ``$\varepsilon$-approximate harmonic maps":
\begin{defn}
Let $0<\varepsilon<1$. We call ``$\varepsilon$-approximate harmonic maps" the states $(u, u_t)(x)\in T\Sph^k$ that belong to the set
\begin{equation}
    \mathcal{Q}_{\varepsilon}:= \bigcup\limits_{\bar \phi(x): \textrm{ a harmonic map }} \left\{(u, u_t): \|(u, u_t)- (\bar \phi, 0)\|_{H^1_x\times L^2_x}\leq \varepsilon\right\}.
\end{equation}
\end{defn}

The first result shows that for any given $\varepsilon$ and any given initial state, the unique solution of the damped wave maps equation  becomes a ``$\varepsilon$-approximate harmonic map" after a long time evolution.
\begin{thm}\label{thm-converharm}
Let $\varepsilon>0$ and $M>0$. There exists some effectively computable $T_{\varepsilon, M}>0$ such that for any initial state  $u[0]\in H^1_xL^2_x(\mathbb{S}^1;T\mathbb{S}^k)$ in $\textbf{H}(M)$,
there exists some time $t_0\in [0,  T_{\varepsilon, M}]$ such that
the unique solution of the damped wave maps equation enters the ``$\varepsilon$-approximate harmonic maps" region:
\begin{equation}
    \phi[t_0]\in \mathcal{Q}_{\varepsilon} \cap \textbf{H}(M). \notag
\end{equation}
\end{thm}

In the second stage, for any ``$\varepsilon$-approximate harmonic maps" we  construct an explicit control to lower its energy below the critical value provided $\varepsilon$ is sufficiently small:
\begin{thm}\label{thm-decreseenergy}
Let $T= 2\pi$, $M>0$. There exist some effectively computable constants $\nu_1>0$, $\varepsilon_1>0$, $C_1>0$ such that, for any initial state $u[0]$ such that
\begin{equation*}
    u[0]\in \mathcal{Q}_{\nu_1}\cap \textbf{H}(M),
\end{equation*}
we can construct  an explicit control $\bar f(t, x)$ compactly supported in $[0, T]\times \omega$ satisfying
\begin{equation*}
  \|\bar f\|_{L^{\infty}_t L^2_x([0, T]\times \Sph^1)}\leq C_1,
\end{equation*}
and find some integer $N$ satisfying
\begin{equation*}
    2\pi N^2\leq M,
\end{equation*}
such that
the unique solution of
\begin{equation*}
\begin{cases}
     \Box \phi= \left(|\phi_{t}|^2- |\phi_{ x}|^2\right) \phi+
    f^{\phi^{\perp}},\\ \phi[0](x)= u[0](x),
    \end{cases}
\end{equation*}
verifies
\begin{equation}
    E(T)\in (2\pi N^2- 10 \varepsilon_1^2, 2\pi N^2-  \varepsilon_1^2).  \notag
\end{equation}
\end{thm}

Armed with the preceding theorems,  we can easily prove Theorem \ref{THM-globalexact}, while the rest of this section is devoted to the proof of the two preceding theorems.
\begin{proof}[Proof of Theorem \ref{THM-globalexact}]
    This is a direct consequence of Theorems \ref{thm-converharm}--\ref{thm-decreseenergy} and the low-energy exact controllability property, Proposition \ref{thm:semicontrol} $(ii)$.
\end{proof}

\subsection{Proof of Theorem \ref{thm-converharm}: the damped wave maps flow convergences to harmonic maps}

We first focus in this section on the proof of Theorem \ref{thm-converharm}  showing that for any given initial state the damped wave maps flow converges to some harmonic map.
Indeed, the required result is a direct consequence of Proposition \ref{thm:semicontrol} $(i)$ and the following intermediate result.
% Before stating this intermediate result, let us  recall that there exists some effectively computable constant $Q\geq 1$  depending only on the value of  $\|a(x)\|_{L^{\infty}(\Sph^1)}$ such that for any solution of the damped wave maps equation there is
% \begin{equation*}
%     Q^{-1}E(0)\leq E(16\pi).
%\end{equation*}

% \begin{prop}\label{prop:23}
% Let $E_M> 0$. For any $\delta\in (0, Q^{-1})$, there exists some effectively computable constant  $\varepsilon_1(\delta)>0$ such that, if some solution of the damped wave maps equation verifies
% \begin{equation*}
%     1\leq E(0)\leq E_M,
% \end{equation*}
% then, either
% \begin{equation}\label{eq:cho1}
%     \int_0^{32\pi} \int_{\Sph^1} a(x) |\phi_t|^2(t, x) dx dt\geq \delta,
% \end{equation}
% or
% \begin{equation}\label{eq:cho2}
%     \exists\; \bar t\in [0, 32\pi] \textrm{ such that } \phi[\bar t]\in \mathcal{Q}_{\varepsilon_1(\delta)}.
% \end{equation}
% \end{prop}

\begin{prop}\label{prop:23}
Let $M> 0$. There exist some effectively computable integer $p>0$ and constants $\varepsilon_0>0$ and $C_0>0$ such that  for any $\varepsilon\in (0, \varepsilon_0)$,
%by selecting  $\delta_1= \delta_1(\varepsilon):= C_0 \varepsilon^{4p}$ we know that,
if some wave map $\phi[t]|_{t\in [0, 32\pi]}$, solution of the system \eqref{eq:dwm}, verifies
\begin{equation*}
    1\leq E(0)\leq M,
\end{equation*}
then, either
\begin{equation}\label{eq:cho1}
    \int_0^{32\pi} \int_{\Sph^1} a(x) |\phi_t|^2(t, x) dx dt\geq \delta_1 \; \textrm{ with } \delta_1= \delta_1(\varepsilon):= C_0 \varepsilon^{4p},
\end{equation}
or
\begin{equation}\label{eq:cho2}
    \exists\; \bar t\in [0, 32\pi] \textrm{ such that } \phi[\bar t]\in \mathcal{Q}_{\varepsilon}.
\end{equation}
\end{prop}

Before giving the proof of Proposition \ref{prop:23} we shall recall the following lemma demonstrated in \cite{Krieger-Xiang-2022}.  Actually, in \cite{Krieger-Xiang-2022} the authors only proved the result for  the special case that  $M= 2\pi$, while the same method also works for the general energy upper bound $M$.
\begin{lem}[\cite{Krieger-Xiang-2022},  Proposition 2.2]\label{lem:2}
Let $M>0$. There exist some effectively computable  integer $p>0$ and constant $C_p>0$ such that, for any $\delta\in (0, 1)$, if  some solution of the damped wave maps equation \eqref{eq:dwm} verifies
\begin{gather}
    E(0)\leq M,\\
\int_{-16\pi}^{16\pi}\int_{\mathbb{S}^1}a(x)|\phi_t|^2(t, x) dxdt\leq \delta  E(0), \label{as:eq:p2}
\end{gather}
then
\begin{equation}
     \|\phi_t\|_{L^{\infty}_x(\Sph^1; L^2_t(0, 3\pi))}^2\leq e_0 E(0),
\end{equation}
where $e_0= e_0(\delta)= C_p \delta^{1/p}$.
\end{lem}

\begin{proof}[Proof of Proposition \ref{prop:23}]
Let us assume that the inequality \eqref{eq:cho1} is false for some $\delta>0 $ with its value to be chosen later on, therefore,
\begin{equation}\label{eq:cho1-f}
    \int_0^{32\pi} \int_{\Sph^1} a(x) |\phi_t|^2(t, x) dx dt< \delta
\end{equation}
In the following we show that by selecting $\delta$ sufficiently small depending on the value of $\varepsilon$, the alternative \eqref{eq:cho2} ought to be true.    One immediately infers from the assumption \eqref{eq:cho1} and the energy identity
\begin{equation*}
    E(t)= E(0)- 2\int_0^t \int_{\Sph^1} a(x) |\phi_t|^2(s, x) dx ds,
\end{equation*}
that
\begin{equation*}
    E(0)- E(t)\leq 2\delta, \;\forall t\in [0, 32\pi].
\end{equation*}
Using  in particular \eqref{estimate-with-Q},
\begin{gather*}
   Q^{-1}\leq  Q^{-1} E(0)\leq E(16\pi)\leq E(0)\leq M,
\end{gather*}
which, together with \eqref{eq:cho1-f}, gives
\begin{gather*}
    \int_0^{32\pi} \int_{\Sph^1} a(x) |\phi_t|^2 dx dt\leq \delta\leq (Q \delta) E(16 \pi).
\end{gather*}
Then, according to Lemma \ref{lem:2}, by regarding $t= 16\pi$ as the initial time,  we know that
\begin{equation}\label{keyes1}
      \|\phi_t\|_{L^{\infty}_x(\Sph^1; L^2_t(I))}^2\leq e_0 (Q\delta) E(16\pi)\leq M C_p Q^{\frac{1}{p}} \delta^{\frac{1}{p}},
\end{equation}
where the time interval $I$ is defined as
\begin{equation}
 I:= [16\pi, 19\pi].
\end{equation}

Now, following the argument  in \cite{Krieger-Xiang-2022} the proof is composed of 4 steps. As we shall see later, the first 3 steps remain as in \cite{Krieger-Xiang-2022}, while a major difference appears in the last step. \\

\noindent {\bf Step 1.}  Let  $\psi(t)$ be a smooth non-negative bump function satisfying
 \begin{align*}
 \int_{\mathbb{R}}\psi(t)\,dt = 1,\,\text{supp}(\psi)\subset I.
\end{align*}
Then
\begin{align*}
-\int_I\int_{\Sph^1}[\big|\phi_t\big|^2 + \big|\phi_x\big|^2]\psi(t)\,dx dt + E(0) = \int_I \psi(t)(-E(t) + E(0))\,dt\leq 2 \delta.
\end{align*}
The preceding inequality, when combined with \eqref{keyes1}, implies that
\begin{align*}
\Big|\int_I\int_{\Sph^1} \big|\phi_x\big|^2\psi(t)\,dx dt - E(0)\Big|\lesssim \delta^{\frac{1}{p}}.
\end{align*}
Thus, thanks to  the intermediate value theorem, one obtains  the existence of some $\bar{x}\in \Sph^1$ such that
\begin{align}\label{bound_step1}
\Big|\int_{I} \big|\phi_x\big|^2(t, \bar{x})\psi(t)\, dt - \frac{E(0)}{2\pi}\Big|\lesssim \delta^{\frac{1}{p}}.
\end{align}

\noindent {\bf Step 2.}  Now, we claim that such a  bound in \eqref{bound_step1} can be ``propagated" to arbitrary $x_1\in \Sph^1$, in a slightly weaker form with $ \delta^{\frac{1}{p}}$ replaced by $ \delta^{\frac{1}{2p}}$. In fact, for any given  $x_1\in \Sph^1$, we observe that
\begin{align*}
&-\int_{I} \big|\phi_x\big|^2(t, \bar{x})\psi(t)\, dt + \int_{I} \big|\phi_x\big|^2(t, x_1)\psi(t)\, dt\\
& = 2\int_{\bar{x}}^{x_1}\int_I\phi_{xx}\cdot\phi_x\psi(t)\,dt dx\\
& = 2\int_{\bar{x}}^{x_1}\int_I\big[\phi_{tt} + a(x)\phi_t\big]\cdot\phi_x\psi(t)\,dt dx\\
& = 2\int_{\bar{x}}^{x_1}\int_I\big[-\phi_t\cdot\phi_x\psi'(t) - \phi_t\cdot\phi_{tx}\psi(t) + a(x)\phi_t\cdot\phi_x\psi(t)\big]\,dt dx.
\end{align*}
Then, we can individually bound
\begin{gather*}
\Big|2\int_{\bar{x}}^{x_1}\int_I\phi_t\cdot\phi_x\psi'(t)\,dt dx\Big|\lesssim \big\|\phi_t\big\|_{L_x^\infty L_t^2(I)} \big\|\phi_x\big\|_{L_x^2 L_t^2(I)}\lesssim \delta^{\frac{1}{2p}}, \\
 \Big| 2\int_{\bar{x}}^{x_1}\int_I a(x)\phi_t\cdot\phi_x\psi(t)\,dt dx\Big|\lesssim \big\|\phi_t\big\|_{L_x^\infty L_t^2(I)} \big\|\phi_x\big\|_{L_x^2 L_t^2(I)}\lesssim \delta^{\frac{1}{2p}},
\end{gather*}
as well as (using integration by parts with respect to $x$ variable),
 \begin{align*}
 \Big| 2\int_{\bar{x}}^{x_1}\int_I\left(\phi_t\cdot\phi_{tx}\right) \psi(t)\,dt dx\Big|&= \int_I |\phi_t(t, x_1)|^2- |\phi_t(t, \bar x)|^2 dt \\
 &\lesssim \big\|\phi_t\big\|_{L_x^\infty L_t^2(I)}^2\lesssim \delta^{\frac{1}{p}}.
 \end{align*}
 Combining these bounds, we infer that
 \begin{equation}\label{eq:phixuniformbound}
 \sup_{x_1\in \Sph^1}\Big|\int_{I} \big|\phi_x\big|^2(t, x_1)\psi(t)\, dt - \frac{E(0)}{2\pi}\Big|\lesssim \delta^{\frac{1}{2p}}.
 \end{equation}

 \noindent {\bf Step 3.} In this step we show that the following {\it time-independent} function is close to a harmonic map:
 \begin{align*}
 \tilde{\phi}(x) := \int_{\mathbb{R}} \phi(t, x)\psi(t)\,dt.
 \end{align*}

  We  know that  for any  $t_1\in I$
 \begin{align*}
 \int_I \phi(t, x)\psi(t)\,dt &= \phi(t_1, x)  + \int_I [\phi(t, x) - \phi(t_1,x)]\psi(t)\,dt\\
 & =  \phi(t_1, x)  + \int_I \left(\int_{t_1}^t \phi_t(s,x)\right)ds \,\psi(t)\, dt.
 \end{align*}
Then, thanks to \eqref{bound_step1} and \eqref{keyes1}, we find
 \begin{equation}\label{eq:phitildephiproximity}
     \left| \tilde \phi(x)- \phi(t_1, x)\right|\lesssim \|\phi_t\|_{L^{\infty}_x L^2(I)}\lesssim \delta^{\frac{1}{2p}}, \; \forall t_1\in I,\; \forall x\in \Sph^1.
 \end{equation}
 Thus
 \begin{align}
 \Big|\tilde{\phi}(x)\Big| &= 1 + O\big(\big\|\phi_t\big\|_{L_x^\infty L_t^2(I)}\big)\notag\\
 & = 1 + O(\delta^{\frac{1}{2p}})\label{esttildephi}.
 \end{align}
 Therefore, there exists some effectively computable number $\delta_{s1}>0$ such that
 \begin{equation}
   \Big|\tilde{\phi}(x)\Big|\in \left[\frac12, \frac32\right], \; \forall x\in \mathbb{S}^1,   \notag
 \end{equation}
 provided that $\delta$ is smaller than $\delta_{s1}$.\\
%  Moreover, we also record that
%  \begin{equation}\label{eq:phitildephiproximity}
%  \sup_{t\in I}\Big|\tilde{\phi}(x) - \phi(t, x)\Big|\ll \delta^{\frac{1}{2p}}.
%  \end{equation}

We now deduce that $\tilde{\phi}$ approximately solves the harmonic map equation, which has important implications on its structure.
 Note that
 \begin{equation}\label{eq:tildephixx}\begin{split}
 \tilde{\phi}_{xx} &= \int_{\mathbb{R}} \phi_{xx}\psi(t)\,dt\\
 & = \int_{\mathbb{R}} \big(\phi_{tt} - |\phi_x|^2\phi + |\phi_t|^2\phi + a(x)\phi_t\big)\psi(t)\,dt\\
 & = \int_{\mathbb{R}} \big(- |\phi_x|^2\phi \psi(t) - \phi_t\psi_t + |\phi_t|^2\phi\psi(t) + a(x)\phi_t\psi(t)\big)\,dt
 \end{split}\end{equation}
 All terms involving at least one factor $\phi_t$ are small thanks to \eqref{keyes1}. In particular, we obtain
 \begin{align*}
 \Big|\int_{\mathbb{R}} \big(- \phi_t\psi_t + |\phi_t|^2\phi\psi(t) + a(x)\phi_t\psi(t)\big)\,dt\Big|
 \lesssim \delta^{\frac{1}{2p}}
 \end{align*}
 uniformly in $x\in S^1$.
 \\

 Taking advantage of \eqref{eq:phixuniformbound} and \eqref{eq:phitildephiproximity}, we find  that
 \begin{align*}
 \Big|\int_I |\phi_x|^2(t, x)\phi(t, x) \psi(t)\,dt  - \tilde{\phi}(x)\cdot\frac{E(0)}{2\pi}\Big|\lesssim \delta^{\frac{1}{2p}}
 \end{align*}
 for all $x\in S^1$. Indeed, by picking  some point $t_0\in I$ we obtain
 \begin{gather*}
     \left|\int_{\mathbb{R}} \phi \psi |\phi_x|^2 dt- \phi(t_0) \int_{\mathbb{R}} \psi |\phi_x|^2 dt\right|\lesssim \|\phi_t\|_{L^{\infty}_x L^2_t(I)}\lesssim \delta^{\frac{1}{2p}}, \\
      \left|\frac{E(0)}{2\pi}\cdot \int_{\mathbb{R}} \phi \psi dt- \frac{E(0)}{2\pi}\cdot\phi(t_0)\right|\lesssim \|\phi_t\|_{L^{\infty}_x L^2_t(I)}\lesssim \delta^{\frac{1}{2p}}.
 \end{gather*}

 In conclusion, the preceding bounds and \eqref{eq:tildephixx} imply  that  the function $\tilde \phi$ satisfies the following
 \begin{align}
 \Big|\tilde{\phi}_{xx}(x) + \frac{E(0)}{2\pi}\cdot\tilde{\phi}(x)\Big|\lesssim \delta^{\frac{1}{2p}}
 \end{align}
 for all $x\in \Sph^1$.
\\

\noindent {\bf Step 4.}
We can now show that the value of $E(0)$ is close to the discrete set  $\mathcal{A}: = \{2\pi n^2\,|\,n\in \mathbb{N}\}$, with the distance depending on the value of $\delta$.
Denote
\begin{equation}\label{eq:deltastar}
\delta_*: = \text{dist}\big(E(0), \mathcal{A}\big)\geq 0.
\end{equation}
Since $E(0)$ is smaller than $M$, so is $\delta_*$.
Writing
$$f(x): = \tilde{\phi}_{xx}(x) + \frac{E(0)}{2\pi}\tilde{\phi}(x),\; \forall x\in \Sph^1,$$
we can develop the functions  $f$ and  $\tilde{\phi}$ into Fourier series
\begin{equation}\label{eq:Fourierseries}
f(x) = \sum_{n\in \Z}f_n e^{inx} \textrm{ and } \tilde{\phi}(x) = \sum_{n\in \Z}a_n e^{inx},
\end{equation}
where  we have the straightforward bounds
\begin{align}
\label{estfn}
\big|f_n \big|\leq \big\|f\big\|_{L^1(S^1)}\lesssim \delta^{\frac{1}{2p}},\;\forall n\in \Z.
\end{align}
Moreover, since both $\tilde \phi(x)$ and $f(x)$ are real valued,
\begin{equation}
    f_{-n}= \overline{f_n} \textrm{ and } a_{-n}= \overline{a_n}.\notag
\end{equation}
%Note that the superscript $(I)$ refers to the fact that both $\tilde{\phi}, f$ depend on the choice of the interval $I$.
By comparing the coefficients of the preceding equation,  the value of $a_n$ in turn satisfies
\begin{equation}\label{linkaf}
    a_n\big(\frac{E(0)}{2\pi} - n^2\big) = f_n.
\end{equation}
Since the value of $M$ is fixed, we know from the assumption $E(0)\leq M$ that
\begin{gather*}
     \delta_* n^2 \lesssim |E(0)- 2\pi n^2|,  \;  \forall |n|\leq M,    \\
      \delta_* n^2 \lesssim |E(0)- 2\pi n^2|,  \;  \forall |n|> M.
\end{gather*}
Hence,
\begin{equation*}
    \delta_* n^2 \lesssim |E(0)- 2\pi n^2|,  \;  \forall E(0)\leq M,\; \forall n\in \mathbb{Z}.
\end{equation*}
Therefore, by taking advantage of \eqref{eq:deltastar},
\begin{equation}\label{eq:anbounds}
\big|a_n\big|\lesssim \frac{\big|f_n\big|}{\delta_* n^2}\lesssim  \frac{\delta^{\frac{1}{2p}}}{\delta_* n^2}.
\end{equation}
Recalling that  $ \Big|\tilde{\phi}(x)\Big|\in [\frac12, \frac32]$, this implies  that
\begin{align*}
\frac{1}{2}\leq \Big|\tilde{\phi}(x)\Big|\leq \sum_{n\in \Z}\big|a_n\big|\lesssim \frac{\delta^{\frac{1}{2p}}}{\delta_*}, \; \forall x\in \Sph^1.
\end{align*}
Hence
\begin{equation}
    \text{dist}\big(E(0), \mathcal{A}\big)= \delta_*\lesssim \delta^{\frac{1}{2p}}.
\end{equation}
In the following we shall let $\delta$ be smaller than some effectively computable constant $\delta_{s2}$ such that $\delta_*$ is smaller than $1/5$.
\\
% We conclude that
% \begin{align*}
% E_{\infty}\in \mathcal{A},
% \end{align*}
% whence the first assertion of the proposition.
% \\

Now let us assume that $E(0)$ is close to $2\pi n_0^2$ for some non-zero integer $n_0$. Thus from our assumption on $\delta_*$ we derive
\begin{equation*}
    \left|\frac{E(0)}{2\pi}- n^2\right|\geq 1,\; \forall n\neq \pm n_0.
\end{equation*}
Observe that, from \eqref{estfn}  and \eqref{linkaf},
\begin{equation}\label{neqn0small}
    |a_n|\lesssim \frac{\delta^{\frac{1}{2p}}}{n^2},\; \forall n\neq \pm n_0.
\end{equation}
Hence
\begin{align*}
\sum_{n\neq \pm n_{0}}\big|a_n\big|\lesssim \sum_{n\neq \pm n_{0}} \frac{\delta^{\frac{1}{2p}}}{n^2} \lesssim \delta^{\frac{1}{2p}},
\end{align*}
 which, in conjunction with \eqref{esttildephi} as well as \eqref{eq:Fourierseries}, gives
 \begin{align*}
\Big|\big|\sum_{n = \pm n_{0}}a_n e^{inx}\big| - 1\Big|\lesssim \delta^{\frac{1}{2p}}.
 \end{align*}
 Recalling that $a_{n}$ and $a_{-n}$ are conjugate, we write
 $$a_{n_{0}} = \alpha_{0} + i \beta_{0} \textrm{ and } a_{-n_{0}} = \alpha_{0} - i \beta_{0}, \text{ with $\alpha_0$ and $\beta_0$ \text{ in } $\R^{k+1}$}.$$
This gives
 \begin{align}\label{roughestildphi}
 \tilde{\phi}(x) = 2 \alpha_0 \cos n_{0} x  - 2\beta_0 \sin n_{0}x  + O\big(\delta^{\frac{1}{2p}}\big).
 \end{align}
 Thus
 \begin{equation*}
    |2 \alpha_0 \cos n_{0} x  - 2\beta_0 \sin n_{0}x|= 1+ O\big(\delta^{\frac{1}{2p}}\big), \;\forall x\in \Sph^1.
 \end{equation*}
 We easily infer from the preceding equation that
 \begin{align*}
&\big|\alpha_{0}\big| = \big|\beta_{0}\big| + O(\delta^{\frac{1}{2p}}) = \frac12 + O(\delta^{\frac{1}{2p}}),\\
& \alpha_{0}\cdot \beta_{0}= O(\delta^{\frac{1}{2p}}).
 \end{align*}
 It follows that there exist $\mu_{0}\in \R^{k+1}$ and $\nu_{0}\in \R^{k+1}$ satisfying
 $$\mu_{0}= 2\alpha_{0} + O(\delta^{\frac{1}{2p}}), \nu_{0} = 2\beta_{0} + O(\delta^{\frac{1}{2p}}),$$
 such that
 \begin{gather*}
 |\mu_0|= |\nu_0|= 1, \; \mu_0\cdot \nu_0=0.
 \end{gather*}
 Indeed, it suffices to select

\begin{align*}
    \mu_0:= \frac{2\alpha_0}{|2\alpha_0|} \; \textrm{ and } \; \nu_0:= \frac{2\beta_0- 2(\beta_0\cdot \alpha_0) \frac{\alpha_0}{|\alpha_0|^2}}{\left|2\beta_0- 2(\beta_0\cdot \alpha_0) \frac{\alpha_0}{|\alpha_0|^2}\right|}.
\end{align*}
Concerning such a selected pair we know that the curve
 \begin{gather*}
 \gamma_{0}(x): = \mu_{0}\cos n_{0} x -  \nu_{0}\sin n_{0}x
 \end{gather*}
 is a geodesic, since it satisfies the harmonic maps equation
 \begin{equation*}
     \gamma_{0xx}+ |\gamma_{0x}|^2\gamma_0= 0.
 \end{equation*}
 Moreover, thanks to  Equation \eqref{roughestildphi}, it also satisfies
  the proximity condition
 \begin{equation}\label{eq:geodesicproxy1}
 \Big|\tilde{\phi}(x) -  \gamma_{0}(x)\Big|\lesssim \delta^{\frac{1}{2p}}.
 \end{equation}

 The preceding formula, together with \eqref{eq:phitildephiproximity}, gives the $L^{\infty}$-closeness of $\phi(t, x)$ to the geodesic $\gamma_0(x)$:
 \begin{equation}
     |\phi(t, x)- \gamma_0(x)|\lesssim \delta^{\frac{1}{2p}}, \; \forall t\in I, \;\forall x\in \Sph^1.
 \end{equation}

We claim also that there exists some $\bar t\in I$ such that $\phi(\bar t, \cdot)$ is a $O(\delta^{\frac{1}{4p}})-$approximate harmonic map, namely,
\begin{equation*}
    \|(\phi[\bar t]- (\gamma_0(x), 0)\|_{H^1_x\times L^2_x}\lesssim \delta^{\frac{1}{4p}}.
\end{equation*}
 In fact, observe that
 \begin{align*}
\int_I \|\phi_x - \tilde{\phi}_x\|_{L_x^2(\Sph^1)}^2\psi(t)\,dt
&= -\int_I \int_{\Sph^1}(\phi - \tilde{\phi})\cdot (\phi_{xx} - \tilde{\phi}_{xx})\psi(t)\,dx dt\\
& =  -\int_I \int_{\Sph^1}(\phi - \tilde{\phi})\cdot \big[\phi_{tt} + (|\phi_t|^2 - |\phi_x|^2)\phi + a(x)\phi_t - \tilde{\phi}_{xx}\big]\psi(t)\,dx dt.
 \end{align*}
 Then, we estimate
 \begin{align*}
 \Big|\int_I \int_{\Sph^1}(\phi - \tilde{\phi})\cdot \phi_{tt}\psi(t)\,dx dt\Big|
 &\leq \Big|\int_I \int_{\Sph^1}(\phi - \tilde{\phi})\cdot \phi_{t}\psi'(t)\,dx dt\Big|
  + \Big|\int_I \int_{\Sph^1}|\phi_t|^2\psi(t)\,dx dt\Big|\\
 &\lesssim \delta^{\frac{1}{p}},
 \end{align*}
 and further, also keeping \eqref{eq:phitildephiproximity} in mind, we find
 \begin{align*}
 &\Big| \int_I \int_{\Sph^1}(\phi - \tilde{\phi})\cdot \big[(|\phi_t|^2 - |\phi_x|^2)\phi + a(x)\phi_t - \tilde{\phi}_{xx}\big]\psi(t)\,dx dt\Big|\\
 &\leq \big\|\phi - \tilde{\phi}\big\|_{L_{t,x}^\infty(S^1\times I)}\cdot \big[\big\|\nabla_{t,x}\phi\|_{L_t^\infty L_x^2}^2 + \big\|\phi_t\|_{L_t^\infty L_x^2} + \big\|\tilde{\phi}_{xx}\big\|_{L_x^2}\big]\\
 &\lesssim \delta^{\frac{1}{2p}}.
 \end{align*}
 The conclusion is that
 \begin{align*}
 \int_{\mathbb{R}} \|\phi_x - \tilde{\phi_x}\|_{L_x^2(\Sph^1)}^2\psi(t)\,dt\lesssim \delta^{\frac{1}{2p}}.
 \end{align*}
 This, when combined with the inequality \eqref{keyes1}, gives
 \begin{equation*}
     \int_{\mathbb{R}} \left(\|\phi_x - \tilde \phi_x\|_{L_x^2(\Sph^1)}^2\psi(t)+ \|\phi_t\|^2_{L^2_x(\mathbb{S}^1)}\right)\,dt\lesssim \delta^{\frac{1}{2p}},
 \end{equation*}
 which entails the existence of $t_{\delta}\in I$ with the property that
 \begin{equation}\label{eq:small1}
 \|\phi_t(t_{\delta}, \cdot)\|_{L_x^2(\Sph^1)}+ \|(\phi_x - \tilde \phi_x)(t_{\delta}, \cdot)\|_{L_x^2(\Sph^1)}\lesssim \delta^{\frac{1}{4p}}.
\end{equation}

Coming back to \eqref{eq:Fourierseries} and using \eqref{neqn0small}, we also easily obtain
\begin{align*}
\Big\|\tilde{\phi}(x) - \sum_{n = \pm n_{0}}a_n e^{inx}\Big\|_{H^1(S^1)}\lesssim \delta^{\frac{1}{2p}},
\end{align*}
and the discussion preceding the choice of $\gamma_{0}$ then entails that
\begin{align*}
    \|\phi(t_{\delta}, \cdot)- \gamma_0(\cdot)\|_{H^1(\Sph^1)}&\leq   \|\phi(t_{\delta}, \cdot)- \gamma_0(\cdot)\|_{L^2(\Sph^1)}+ \|\phi_x(t_{\delta}, \cdot)- \gamma_{0, x}(\cdot)\|_{L^2(\Sph^1)}, \\
    &= O(\delta^{\frac{1}{4p}})+ \|\tilde \phi_x(\cdot)- \gamma_{0, x}(\cdot)\|_{L^2(\Sph^1)}, \\
    &= O(\delta^{\frac{1}{4p}})+ \| \sum_{n = \pm n_{0}}i n a_n e^{inx}- \gamma_{0, x}(\cdot)\|_{L^2(\Sph^1)},\\
    &\lesssim \delta^{\frac{1}{4p}}.
\end{align*}
This, together with \eqref{eq:small1}, implies that there exists some effectively computable constant $c_0$ such that
\begin{align*}
\Big\|\phi[t_{\delta}] - (\gamma_{0}, 0)\Big\|_{H^1(\Sph^1)\times L^2(\Sph^1)}\leq c_0 \delta^{\frac{1}{4p}}.
\end{align*}
Hence, it suffices to set $\varepsilon_0>0$ in such a way that
\begin{equation}
    \left(\frac{\varepsilon_0}{c_0}\right)^{4p}\leq \min\{\delta_{s1}, \delta_{s2}\},
\end{equation}
and $$C_0:= \frac{1}{c_0^{4p}}$$ for the definition of $\delta_1(\varepsilon)$. This ends the proof of Proposition \ref{prop:23}.
\end{proof}

\subsection{Proof of Theorem \ref{thm-decreseenergy}: dissipation of the energy around  harmonic maps}
Armed with Theorem \ref{thm-converharm} we shall only deal with the wave maps solutions around harmonic maps, since otherwise one may add a localized damping force and let the solution dissipate towards some harmonic map.
One easily observes that in our framework the harmonic maps are not local minimizers of the energy, due to the simple geometry of the sphere. As proved in the preceding section, with the help of the damping term the unique solution of the wave maps equation becomes sufficiently close to some harmonic map. As a direct consequence, its energy also approximates $2\pi N^2$ for some integer $N$ and converges toward this value from above. In this section we show that with the help of some other well-designed control the energy of an ``approximate harmonic map" can be decreased and become strictly less than  $2\pi N^2$. Then, we are able to iterate the preceding procedure:   a damping control  forces the energy to dissipate toward $2\pi (N-1)^2$ or a lower value, which, when combined with another well-designed control, will become strictly less than $2\pi (N-1)^2$.\\

In order to simplify the notations, in the rest part of this section we only deal with  the simplest case: assume that the target is $\Sph^2$ and  that the initial  state $\phi[0]$ is close to the geodesic $\bar u[0]= (\bar u(x), 0)$ with
\begin{equation}\label{groundstate}
    \bar u(x):= (\cos{x}, \sin{x}, 0).
\end{equation}
%such that the energy is bigger than $2\pi$.
\begin{remark}\label{remark-rotation}
 Actually, the wave maps equation, the damped wave maps equation as well as the inhomogeneous wave maps equation are invariant under  the action of the orthogonal group. More precisely, suppose that $(\phi, f)$ is a solution of the inhomogeneous wave maps equation,
\begin{equation} \notag
\begin{cases}
     \Box \phi= \left(|\phi_{t}|^2- |\phi_{x}|^2\right) \phi+  f^{\phi^{\perp}},  \\ (\phi, \phi_t)(0, x)= \phi[0],
\end{cases}
\end{equation}
then, for any matrix $A$ belongs to $O(3)$, the pair $(\bar \phi, \bar f):= (A \phi, Af)$ is also a solution of the inhomogeneous wave maps equation:
\begin{equation}\notag
\begin{cases}
     \Box \bar\phi= \left(|\bar\phi_{t}|^2- |\bar\phi_{x}|^2\right) \bar\phi+  \bar f^{\bar\phi^{\perp}},  \\ (\bar\phi, \bar\phi_t)(0, x)= A\phi[0].
\end{cases}
\end{equation}
One also observes that for every harmonic map $x\in \Sph^1\mapsto v(x) \in \mathbb{S}^2\subset \mathbb{R}^3$ having energy $2\pi N^2$, there exists an orthogonal matrix $A\in O(3)$ such that
\begin{equation*}
    A v(x)= \bar u(Nx), \; \forall x\in \mathbb{S}^1.
\end{equation*}
\end{remark}

% Recall  that
% \begin{equation}
%     E(\phi):= \int_{\Sph} |\phi_x|^2+ |\phi_t|^2 dx.
% \end{equation}
\noindent {\bf A straightforward variational point of view.}

Recall the definition of the energy $E$: Equation  \eqref{def-energy}. Around the harmonic map state $\bar u[0]$ the controlled wave maps equation satisfies
\begin{align}\label{ener_directes}
    \frac{1}{2} \frac{d}{dt} E(t)&=  \frac{1}{2} \frac{d}{dt}  \int_{\Sph^1} |\phi_x|^2+ |\phi_t|^2 dx \notag \\
    &=\int_{\Sph^1}  \langle \phi_x, \phi_{xt} \rangle+ \langle \phi_t, \phi_{tt} \rangle dx \notag \\
    &= -\int_{\Sph^1} \langle \phi_t, f^{\phi^{\perp}}\rangle  dx.
\end{align}
Therefore, the first derivative of $E(t)$  at the point $\bar u[0]$ is zero. We continue by computing its second derivative. We get
\begin{align*}
     \frac{1}{2} \frac{d^2}{dt^2} E(t)&= - \int_{\Sph} \langle \phi_{tt}, \bar f\rangle+ \langle\phi_t, f_t\rangle  dx \\
     &= \int_{\Sph} \langle \bar f, \bar f- \phi_{xx}\rangle- \langle \bar f_t, \phi_t\rangle  dx,
\end{align*}
where $\bar f$ refers to $ f^{\phi^{\perp}}$.
Hence at the point $\bar u[0]$ there is $E''\geq 0$.  This naive variational observation seems  to be preventing us from getting local controllability around $\bar u[0]$.

However, from a geometric point view, by ignoring the fact that the flow has to satisfy the wave maps equation, one can easily construct a deformation that passes through the harmonic map. Indeed, this is a consequence of the fact that such a geodesic is not a local minimizer of the energy. Thus it becomes essential to understand whether it is the geometric structure of the wave maps equation that forms an obstruction.
To explore this question,  we first assume that the control is acting on the whole circle $\Sph^1$ (i.e. that $\omega=\Sph^1$) to see whether there exist solutions of the wave maps control system starting from $\bar u[0]$ whose energy goes below the energy of $\bar u[0]$. \\

\noindent {\bf Control of the energy around critical values: a special example with control acting   on the whole circle $\mathbb{S}^1$.}

It is natural to consider a symmetric trajectory of the form
\begin{equation*}
    \phi (t, x)= \left(\sqrt{1- \alpha^2(t)} \cos{x}, \sqrt{1- \alpha^2(t)} \sin{x}, \alpha(t)\right)^T,
\end{equation*}
where $\alpha(0)=\alpha'(0)=0$, as this can lead to a trajectory having energy strictly less than $2\pi$ for some positive time. It remains to see whether such a trajectory can be a flow of the controlled wave maps equation.

By replacing $\alpha(t)$ by $\sin \theta(t)$ one may also
describe the trajectory as
\begin{equation}\label{const-phi-full}
    \phi(t, x)= \left(\cos{\theta(t)} \cos{x}, \cos{\theta(t)} \sin{x}, \sin{\theta(t)}\right)^T.
\end{equation}
In this circumstance the control $f(t, x)$ can be chosen in forms of
\begin{equation*}
    f(t, x)= (g(t) \cos{x}, g(t) \sin{x}, h(t))^T,
\end{equation*}
which is  orthogonal to $\phi(t, x)$ provided that
\begin{equation*}
    g(t) \cos{\theta(t)}+ h(t) \sin{\theta(t)}=0.
\end{equation*}
Inspired by the preceding equation one may further restrict the choice of $f(t, x)$ as follows:
\begin{equation}\label{const-f-full}
    f(t, x)= (-w(t) \sin{\theta(t)} \cos{x}, -w(t) \sin{\theta(t)}\sin{x}, w(t) \cos{\theta(t)})^T.
\end{equation}

It remains to see  whether with an appropriate choice of $w(t)$ and $\theta(t)$ the above constructed pair $(\phi, f)$ as in \eqref{const-phi-full} and \eqref{const-f-full} is a solution of the controlled wave maps equation.
By simple calculation one obtains
\begin{equation*}
    \phi_t=
    \begin{pmatrix}
    -\sin{\theta} \theta_t \cos{x}\\
    -\sin{\theta} \theta_t \sin{x}\\
    \cos{\theta} \theta_t
    \end{pmatrix}, \;
    \phi_{tt}=
    \begin{pmatrix}
    -\left(\cos{\theta} \theta_t^2+ \sin{\theta} \theta_{tt}\right) \cos{x} \\
    -\left(\cos{\theta} \theta_t^2+ \sin{\theta} \theta_{tt}\right) \sin{x} \\
   -\sin{\theta} \theta_t^2+ \cos{\theta} \theta_{tt}
    \end{pmatrix},
\end{equation*}
\begin{equation*}
    \phi_x=
    \begin{pmatrix}
    -\cos{\theta} \sin{x}\\
    \cos{\theta} \cos{x}\\
    0
    \end{pmatrix}, \;
    \phi_{xx}=
    \begin{pmatrix}
    -\cos{\theta} \cos{x} \\
    -\cos{\theta} \sin{x} \\
    0
    \end{pmatrix},
\end{equation*}
which implies that
\begin{gather*}
    |\phi_t|^2- |\phi_x|^2= \theta_t^2- (\cos{\theta})^2.
\end{gather*}
Thus, $(\phi, f)$ is a trajectory of the controlled wave maps equation if and only if
\begin{equation*}
    \begin{pmatrix}
    \left(\cos{\theta} \theta_t^2+ \sin{\theta} \theta_{tt}\right) \cos{x}- \cos{\theta} \cos{x} - \left(\theta_t^2- (\cos{\theta})^2\right) \cos{\theta} \cos{x} \\
    \left(\cos{\theta} \theta_t^2+ \sin{\theta} \theta_{tt}\right) \sin{x}- \cos{\theta} \sin{x} - \left(\theta_t^2- (\cos{\theta})^2\right) \cos{\theta} \sin{x} \\
    -\cos{\theta} \theta_{tt}+ (\cos{\theta})^2 \sin{\theta}
    \end{pmatrix}
    =
    \begin{pmatrix}
    -w \sin{\theta} \cos{x} \\
    -w \sin{\theta} \sin{x} \\
    w \cos{\theta}
    \end{pmatrix},
\end{equation*}
which is further equivalent to
\begin{equation}
\begin{cases} \notag
\sin{\theta} (\theta_{tt}- \sin{\theta} \cos{\theta})= -w \sin{\theta}, \\
     -\cos{\theta} \theta_{tt}+ (\cos{\theta})^2 \sin{\theta}= w \cos{\theta}.
\end{cases}
\end{equation}
Therefore, it suffices to set
\begin{equation}
    w(t):= -\theta_{tt}+ \sin{\theta} \cos{\theta}= -\frac{1}{2}\left(2\theta_{tt}- \sin 2\theta    \right).\notag
\end{equation}

In conclusion we observe that for any given time dependent function $\theta(t)$ the pair $(\phi, f)$ given by
\begin{equation}\notag
    \phi(t, x)=
    \begin{pmatrix}
    \cos{\theta(t)} \cos{x} \\ \cos{\theta(t)} \sin{x} \\ \sin{\theta(t)}
    \end{pmatrix} \; \textrm{ and } \;
     f(t, x)=
    \begin{pmatrix}
   -w(t) \sin{\theta(t)} \cos{x} \\
   -w(t) \sin{\theta(t)}\sin{x} \\
   w(t) \cos{\theta(t)}
    \end{pmatrix}
\end{equation}
is a solution of the controlled wave maps equation
\begin{equation*}
    \Box \phi= (|\phi_t|^2- |\phi_x|^2)u+ f^{\phi^{\perp}}.
\end{equation*}
More importantly,  even with the constraint $\theta(0)= \theta'(0)=0$, for any $T>0$, there exists $\theta(t)$ such that $u[T]< 2\pi$. Indeed, since
\begin{equation*}
    E(t)= 2\pi \left(\theta_t^2+ (\cos{\theta})^2\right)(t),
\end{equation*}
it suffices to choose $\theta\in C^1([0,T];\R)$ such that  $\theta(0)=\theta'(0)=\theta'(T) =0$ and $\theta(T)\in (0,\pi)$.
In conclusion we have  constructed a radial solution starting at time $0$ from $((\cos(x),\sin(x))^T,(0,0)^T)$ whose energy at a given time $T>0$ is strictly less than $2\pi$.

The above construction of radial solutions indicates that the critical energy value $2\pi$ is not a local minimum value with respect to time for the controlled wave maps equation. It remains to understand  the system with localized control, namely with control which is supported on a maybe small non empty open subset $\omega$ of $\Sph^1$.\\

\noindent {\bf A power series expansion argument to decrease the energy.}
\\

% {\color{magenta}XXXXX This is a weaker but natural version, to be removed.}
% \begin{thm}
% Let $T= 2\pi$. There exists some effectively computable $\varepsilon_1>0$, $\delta_1>0$ and $C_1>0$ such that, for any $c\in [0, 2\pi]$ there exists some explicit control $f(c; x)$ compactly supported in $\omega$ verifying
% \begin{equation*}
%     \|f\|_{L^{\infty}_t L^2_x([0, T]\times \Sph^1)}\leq C_1
% \end{equation*}
% such that for any initial state $u[0]$ verifying
% \begin{equation}
%      \|(u, u_t)(x)- (\bar u, 0)(x+c)\|_{H^1_x\times L^2_x}\leq \delta_1,
% \end{equation}
% the unique solution of
% \begin{equation*}
%     \Box u= \left(|u_{t}|^2- |u_{ x}|^2\right) u+
%      f(c)^{u^{\perp}}, (u, u_t)(0)= u[0],
% \end{equation*}
% satisfies
% \begin{equation}
%     E(T)\leq 2\pi- \varepsilon_1.
% \end{equation}
% \end{thm}

%  Indeed, as we will see in the prof of Proposition \ref{prop_decrea},  the choice of $f$ does not depend on the explicite value of $c\in [0, 2\pi]$. Thus, actually, we can prove the following stronger result:

Now, we are in position to prove Theorem \ref{thm-decreseenergy}, where it suffices to  show the following proposition concerning the simplest harmonic map $\bar u[0]$.  The proof is based on the so-called power series expansion method which is introduced in \cite{Coron-Crepeau} for the local exact control of KdV equations with critical length. (Note that this method can be used together with a change of time-scale in connection with the WKB method as shown in \cite{2022-Coron-Xiang-Zhang-JDE}.)

\begin{prop}\label{prop-pse-decay}
Let $T= 2\pi$. There exist some effectively computable $\varepsilon_0>0$, $\nu_0>0$, $C_0>0$,  and an explicit control $\bar f(t, x)$ compactly supported in $[0, T]\times \omega$ verifying
\begin{equation*}
  \|\bar f\|_{L^{\infty}_t L^2_x([0, T]\times \Sph^1)}\leq C_0,
\end{equation*}
such that for any $A\in O(3)$, for any initial state $u[0]$ verifying
\begin{equation}\label{cond-es-u0}
     \|u[0](x)- A \bar u[0] \|_{H^1_x\times L^2_x}\leq \nu_0,
\end{equation}
the unique solution of
\begin{equation*}
\begin{cases}
     \Box \phi= \left(|\phi_{t}|^2- |\phi_{ x}|^2\right) \phi+
    (A \bar f)^{\phi^{\perp}},\\ \phi[0](x)= u[0](x),
    \end{cases}
\end{equation*}
satisfies
\begin{equation}\label{es-per-2pi}
    E(T)\in( 2\pi- 10\varepsilon_0^2, 2\pi-  \varepsilon_0^2).
\end{equation}
\end{prop}

\begin{proof}[Proof of Proposition \ref{prop-pse-decay}]
To simplify the notations,  thanks to  Remark \ref{remark-rotation}, in this proof we only deal with the  case that $A= Id$, keeping in mind that %exactly the same proof also applies for any $c\in [0, 2\pi]$, and that
the choice of $\varepsilon_0, \nu_0$ and $C_0$ is independent of the rotation $A$. More precisely, the proof is composed of two steps:  first, we assume that the initial state is exactly $\bar u[0]$ and, based on the power series expansion method, we construct an explicit control, $\bar f(t, x)$, to decrease the energy of the system below the critical level $2\pi$; then,  a standard perturbation argument based on Lemma \ref{lem-conti-dep-inh}  implies that for  initial state that  is sufficiently close to $\bar u[0]$ (even if the energy is strictly larger  than $2\pi$),  the preceding designed control $\bar f(t, x)$ still allow to  decreases the energy below the critical level $2\pi$.\\

\noindent {\bf Step 1. A power series expansion argument to dissipate the energy for harmonic maps.}
\\

\begin{prop}\label{prop_decrea}
Let $T= 2\pi$. There exist an effectively computable constant $\varepsilon_0>0$ and an  explicit function $f_1(t, x)$  supported in $[0, T]\times \omega$ satisfying
\begin{equation*}
    \|f_1\|_{L^{\infty}_t L^2_x([0, T]\times \Sph^1)}\leq C,
\end{equation*}
such that for any $\varepsilon\in (0, \varepsilon_0]$,
the unique solution of the inhomogeneous wave maps equation
\begin{equation}\label{eq_full}
\begin{cases}
    \Box \bar \phi= \left(|\bar\phi_{t}|^2- |\bar\phi_{ x}|^2\right)\bar \phi+
     (\varepsilon f_1)^{\bar\phi^{\perp}},\\
   \bar  \phi[0](x)= \bar u[0](x),
\end{cases}
\end{equation}
satisfies
\begin{equation}  \notag
    E(T)\in (2\pi- 3\pi \varepsilon^2, 2\pi- \pi \varepsilon^2).
\end{equation}
\end{prop}

% {\color{blue} As we will see in the following proof, the choice of $\varepsilon_0$ and $f$ does not depend on the value of $c\in [0, 2\pi]$, thus we actually have proved the following stronger result.
% \begin{cor}
% Let $T= 2\pi$. For any $c\in [0, 2\pi]$, for any $\varepsilon\in (0, \varepsilon_0]$,
% the unique solution of
% \begin{equation}\label{eq_full}
% \begin{cases}
%     \Box \phi= \left(|\phi_{t}|^2- |\phi_{ x}|^2\right) \phi+
%      (\varepsilon f_1)^{\phi^{\perp}},\\
%      \phi[0](x)= \bar u[0](x+c),
% \end{cases}
% \end{equation}
% satisfies
% \begin{equation}
%     E(T)\leq 2\pi- \pi \varepsilon^2.
% \end{equation}
% \end{cor}

% }

\begin{proof}[Proof of Proposition \ref{prop_decrea}]
Considering that Equation \eqref{eq_full} is a semilinear equation with geometric constraint, we perform a {\it formal} power series expansion on $ \bar\phi$ and $f$:
\begin{align}
    \bar \phi&= \bar\phi_0+ \varepsilon \bar\phi_1+ \varepsilon^2 \bar\phi_2+... \;  \textrm{ and } \; f:= \varepsilon f_1,
\end{align}
and further denote
\begin{equation}
    \bar\phi_i= ( \bar\phi_i^1,  \bar\phi_i^2,  \bar\phi_i^3)^T \;  \textrm{ and } \; f_1= (0, 0, f_1^3)^T.
\end{equation}

For the zeroth order, we immediately get
\begin{equation}\label{eq_full_1st}
    \Box  \bar\phi_0= \left(| \bar\phi_{0t}|^2- | \bar\phi_{0x}|^2\right)  \bar\phi_0, \; ( \bar\phi_0,  \bar\phi_{0t})(0)= \bar u[0],
\end{equation}
thus $ \bar\phi_0[t]= \bar u[0]$.

For the first order, the equation of $ \bar\phi_1$ reads as
\begin{align}\label{eq:barphi1}
\begin{cases}
    \Box  \bar\phi_1&= \left(2  \bar\phi_{0t}  \bar\phi_{1t}- 2  \bar\phi_{0x}  \bar\phi_{1x}\right)u_0+ \left(| \bar\phi_{0t}|^2- | \bar\phi_{0x}|^2\right)  \bar\phi_1+ f_1- (f_1 \cdot  \bar\phi_0)  \bar\phi_0 \\
    &= - 2 ( \bar\phi_{0x} \cdot  \bar\phi_{1x})  \bar\phi_0-  \bar\phi_1+ f_1- (f_1 \cdot  \bar\phi_0)  \bar\phi_0, \\
     \bar\phi_1[0]&= (0, 0).
     \end{cases}
\end{align}
 Thanks to the choice of $f_1$ we know that
 \begin{equation}\notag
     \begin{cases}
      \Box  \bar\phi^1_1+  \bar\phi^1_1+ 2\left(-(\sin{x}) ( \bar\phi^1_1)_x+ (\cos{x}) ( \bar\phi^2_1)_x\right) \cos{x}= 0, \\
      \Box  \bar\phi^2_1+  \bar\phi^2_1+ 2\left(-(\sin{x})( \bar\phi^1_1)_x+ (\cos{x})( \bar\phi^2_1)_x\right) \sin{x}= 0, \\
      \Box  \bar\phi^3_1+  \bar\phi^3_1= f^3_1,
     \end{cases}
 \end{equation}
 thus $ \bar\phi^1_1(t)=  \bar\phi^2_1(t)= 0$. Concerning the third direction $ \bar\phi^3_1$, we recall the following classical lemma.
 \begin{lem}
 Let $T= 2\pi$. There exists a control $g\in L^{\infty}(0, T; L^2(\Sph^1))$ such that
 the unique solution $\bar v$ of the scalar wave equation
 \begin{equation}\label{eq_v0}
     \Box \bar v+ \bar v= g \textrm{ with } (\bar v, \bar v_t)(0)= (0, 0)
 \end{equation}
 satisfies
 \begin{equation}\label{eq_v1}
     (\bar v, \bar v_t)(T, x)= (-1, 0), \; \forall x\in \Sph^1.
 \end{equation}
 \end{lem}

 Armed with the preceding lemma, in the following we shall directly set $f^3_1:= g$ leading to
 \begin{equation} \notag
      \bar\phi_1^3(t, x)= \bar v(t, x), \;  \textrm{ thus }  \bar\phi^3_1[T]= (-1, 0).
 \end{equation}
 It is natural to  calculate the energy of the first two terms in the power series expansion
 \begin{equation} \notag
      (\bar\phi_0+ \varepsilon  \bar\phi_1)(t, x)= (\cos{x}, \sin{x}, \varepsilon \bar v(t, x))^T,
 \end{equation}
 which turns out to be exactly $2\pi$ at time $T$ (and it is even larger than $2\pi$ for some $t\in (0, 2\pi)$). Hence, it is not  clear  whether the energy of the full controlled system dissipates or not.
 Recall the energy estimate given by \eqref{ener_directes}:
 \begin{equation}  \notag
     E(T)- E(0)= -2\int_0^T \int_{\Sph^1}  \bar\phi_t \cdot (\varepsilon f_1)^{ \bar\phi^{\perp}} dx dt.
 \end{equation}
 Since
 \begin{align*}
      \bar\phi_t \cdot (\varepsilon f_1)^{ \bar\phi^{\perp}}&= ( \bar\phi_{0t}+ \varepsilon  \bar\phi_{1t}+ \varepsilon^2  \bar\phi_{2t}+...)\cdot \left(\varepsilon f_1- \langle \varepsilon f_1,  \bar\phi_0+ \varepsilon  \bar\phi_1+...\rangle (\phi_0+ \varepsilon \phi_1+...)\right) \\
     &= \varepsilon^2  \bar\phi_{1t}\cdot f_1 + O(\varepsilon^3) \\
     &= \varepsilon^2 \bar v_t \cdot g + O(\varepsilon^3),
 \end{align*}
 at least formally, we obtain
 \begin{equation}\label{ene-decay}
     E(T)- E(0)= -2\varepsilon^2 \int_0^T \int_{\Sph^1} \bar v_t   g dx dt+ O(\varepsilon^3).
 \end{equation}

 Together with the precise error estimates that will be proved later on,  it suffices to estimate the integral of $\bar v_t g$, which can be calculated by considering the scalar wave equation
\eqref{eq_v0}--\eqref{eq_v1}. Let
\begin{equation}
    F(t):= \int_{\Sph^1} (|\bar v_x|^2+ |\bar v_t|^2- \bar v^2)(t, x) dx, \;\forall t\in [0, T]. \notag
\end{equation}
 Then
 \begin{align*}
     \frac{1}{2} \frac{d}{dt} F(t)&=  \int_{\Sph^1} \bar v_x \bar v_{xt}+ \bar v_t \bar v_{tt}- \bar v \bar  v_t dx\\
     &= \int_{\Sph^1} \bar v_t (\bar v- g)-\bar v \bar v_t dx \\
     &= - \int_{\Sph^1} \bar v_t g dx,
 \end{align*}
 which implies that
 \begin{equation}  \notag
     -2 \int_0^T \int_{\Sph^1} \bar v_t g dx dt= F(T)- F(0)= -2\pi.
 \end{equation}
 As a direct consequence, using also \eqref{ene-decay}, we know that
 \begin{equation}
     E(T)- E(0)= -2\pi \varepsilon^2+ O(\varepsilon^3). \notag
 \end{equation}

Suppose that the formal energy estimate  \eqref{ene-decay} holds,  then we are able to find some effectively computable $\varepsilon_0$ such that for any $\varepsilon\in (0, \varepsilon_0]$ the unique solution of \eqref{eq_full} with $f:= \varepsilon (0, 0, g)^T$ satisfies
 \begin{equation}\label{decay-u0es}
     E(T)\in (2\pi- 3\pi \varepsilon^2, 2\pi- \pi \varepsilon^2).
 \end{equation}

 To finish the proof, it only remains to present the explicit error estimates on \eqref{ene-decay}.
Let us define the higher order remainder
\begin{equation}  \notag
    w:=  \bar \phi- \bar \phi_0- \varepsilon \bar \phi_1.
\end{equation}
First, thanks to Lemma \ref{lem:inhwm} and Lemma \ref{lem:freewave}, we have the following basic estimates on $\bar \phi, \bar \phi_0, \bar \phi_1$ and $w$ in the domain $D:= [0, T]\times \mathbb{S}^1$:
\begin{gather*}
    \|(\bar \phi, \bar \phi_0, \bar \phi_1, w)\|_{L^{\infty}_{t, x}}\lesssim 1, \\
    \|(\bar \phi_v, \bar \phi_{0v}, \bar \phi_{1v}, w_v)\|_{L^2_v L^{\infty}_u\cap L^{\infty}_u L^2_v}+  \|(\bar \phi_u, \bar \phi_{0u}, \bar \phi_{1u}, w_u)\|_{L^2_u L^{\infty}_v\cap L^{\infty}_v L^2_u}\lesssim 1.
\end{gather*}

In order to improve the estimates on $w$, we define $\bar \phi$, $\bar \phi_0$, and  $\bar \phi_1$ to be the solutions of the following Cauchy problems:
\begin{align*}
    \Box \bar \phi&= \left(\bar \phi_u\cdot \bar \phi_v\right) \bar \phi+ (\varepsilon f_1)^{\bar \phi^{\perp}}, \;\;   (\bar \phi, \bar \phi_t)(0)= \bar u[0], \\
    \Box \bar \phi_0&=  \left(\bar \phi_{0u}\cdot \bar \phi_{0v}\right) \bar \phi_0, \;\;   (\bar \phi_0, \bar \phi_{0t})(0)= \bar u[0], \\
  \Box \bar \phi_1&= \left(\bar \phi_{1u}\cdot \bar \phi_{0v}+ \bar \phi_{0u}\cdot \bar \phi_{1v} \right) \bar \phi_0- \bar \phi_1+ f_1- (f_1 \cdot \bar \phi_0) \bar \phi_0, \;\;  (\bar \phi_1, \bar \phi_{1t})(0)= (0, 0),
\end{align*}
Then  $w$ satisfies the following Cauchy problem:
\begin{equation}  \notag
    w[0]= (0, 0),
\end{equation}
\begin{align*}
    \Box w&= \Box \bar \phi- \Box \bar \phi_0 - \varepsilon \bar \phi_1\\
    &= \left(\bar \phi_u\cdot \bar \phi_v\right) \bar \phi+ (\varepsilon f_1)^{\bar \phi^{\perp}}-\left(\bar \phi_{0u}\cdot \bar \phi_{0v}\right) \bar \phi_0
    \\&-  \varepsilon \Big( \left(\bar \phi_{1u}\cdot \bar \phi_{0v}+ \bar \phi_{0u}\cdot \bar \phi_{1v} \right) \bar \phi_0- \bar \phi_1+ f_1- (f_1 \cdot \bar \phi_0) \bar \phi_0\Big) \\
    &= \langle \bar \phi_{0u}+\varepsilon \bar \phi_{1u}+ w_u, \bar \phi_{0v}+ \varepsilon \bar \phi_{1v}+ w_v\rangle (\bar \phi_0+ \varepsilon \bar \phi_1+ w)- \langle \bar \phi_{0u}, \bar \phi_{0v}\rangle \bar \phi_0\\
    &\;\;\;\;\;\;\;\;  -   \varepsilon \Big( \left(\bar \phi_{1u}\cdot \bar \phi_{0v}+ \bar \phi_{0u}\cdot \bar \phi_{1v} \right) \bar \phi_0+ (\bar \phi_{0u}\cdot \bar \phi_{0v}) \bar \phi_1\Big) \\
    &\;\;\;\;\;\;\;\; - \varepsilon \langle f_1, w+ \bar \phi_0+ \varepsilon \bar \phi_1\rangle (w+ \bar \phi_0+ \varepsilon \bar \phi_1)+ \varepsilon \langle f_1, \bar \phi_0\rangle \bar \phi_0 \\
    &= \Big(w_u\cdot w_v+ w_u\cdot (\bar \phi_{0v}+ \varepsilon \bar \phi_{1v})+ w_v \cdot (\bar \phi_{0u}+ \varepsilon \bar \phi_{1u})\Big) \bar \phi+ \langle \bar \phi_{0u}+ \varepsilon \bar \phi_{1u}, \bar \phi_{0v}+ \varepsilon \bar \phi_{1v}\rangle w \\
    &\;\;\;\;\;\;\;\;\;\;\;\; - \varepsilon (f_1\cdot \bar \phi) w- \varepsilon (f_1\cdot w) (\bar \phi_0+ \varepsilon \bar \phi_1)+ \varepsilon^2 \Big( (\bar \phi_{1u}\cdot \bar \phi_{1v}) \bar \phi_0+ (\bar \phi_{1u}\cdot \bar \phi_{0v}+ \bar \phi_{0u}\cdot \bar \phi_{1v}) \bar \phi_1
    \\
    &\;\;\;\;\;\;\;\;\;\;\;\;+ \varepsilon (\bar \phi_{1u}\cdot \bar \phi_{1v}) \bar \phi_1- (f_1\cdot \bar \phi_1) \bar \phi_0 - (f_1\cdot (\bar \phi_0+ \varepsilon \bar \phi_1)) \bar \phi_1\Big)\\
    &= N(w)+ R,
\end{align*}
where
\begin{align*}
    N(w)&:= \Big(w_u\cdot w_v+ w_u\cdot (\bar \phi_{0v}+ \varepsilon \bar \phi_{1v})+ w_v \cdot (\bar \phi_{0u}+ \varepsilon \bar \phi_{1u})\Big) \bar \phi+ \langle \bar \phi_{0u}+ \varepsilon \bar \phi_{1u}, \bar \phi_{0v}+ \varepsilon \bar \phi_{1v}\rangle w \\
    &\;\;\;\;\;\;\;\;\;\;\;\; - \varepsilon (f_1\cdot \bar \phi) w- \varepsilon (f_1\cdot w) (\bar \phi_0+ \varepsilon \bar \phi_1), \\
    R&:= \varepsilon^2 \Big( (\bar \phi_{1u}\cdot \bar \phi_{1v}) \bar \phi_0+ (\bar \phi_{1u}\cdot \bar \phi_{0v}+ \bar \phi_{0u}\cdot \bar \phi_{1v}) \bar \phi_1+ \varepsilon (\bar \phi_{1u}\cdot \bar \phi_{1v}) \bar \phi_1- (f_1\cdot \bar \phi_1) \bar \phi_0
    \\
    &\phantom{:= \varepsilon^2 \Big(}- (f_1\cdot (\bar \phi_0+ \varepsilon \bar \phi_1)) \bar \phi_1\Big).
\end{align*}

In analogy to the proof of Lemma \ref{lem-conti-dep-inh}, we use a bootstrap argument to estimate $w$. Recalling the notations of $Q_{T_1}$ and $\mathcal{W}_{T_1}$ given in  the proof of Lemma \ref{lem-conti-dep-inh},  thanks to Lemma \ref{lem:freewave}, we obtain
\begin{align*}
    \|w\|_{\mathcal{W}_{T_1}}&\lesssim \|w[0]\|_{H^1\times L^2}+ T_1^{\frac{1}{2}} \|N(w)+ R\|_{L^2_{t, x}(Q_{T_1})}, \\
    &\lesssim \|w[0]\|_{H^1\times L^2}+ T_1^{\frac{1}{2}} \Big( \|w\|_{\mathcal{W}_{T_1}}^2+ \|w\|_{\mathcal{W}_{T_1}}+ \varepsilon \|w\|_{\mathcal{W}_{T_1}}\Big)+ T_1^{\frac{1}{2}} \varepsilon^2, \\
     &\lesssim \|w[0]\|_{H^1\times L^2}+ T_1^{\frac{1}{2}}  \|w\|_{\mathcal{W}_{T_1}}+ T_1^{\frac{1}{2}} \varepsilon^2.
\end{align*}
Hence, by choosing $T_1$ small enough, we obtain
\begin{equation*}
    \|w\|_{\mathcal{W}_{T_1}}\lesssim \|w[0]\|_{H^1\times L^2}+  T_1^{\frac{1}{2}} \varepsilon^2.
\end{equation*}

By iterating this argument, we get
\begin{equation}
    \|w\|_{\mathcal{W}_{T}}\lesssim  \varepsilon^2.
\end{equation}

 Now we come back to the strict estimate for the variation of the energy. Observe that
 \begin{align*}
 & \;\;\;\;  \bar \phi_t \cdot f^{\bar \phi^{\perp}}- \varepsilon^2 \bar \phi_{1t}\cdot f_1^{\bar \phi_0^{\perp}} \\
 &= \bar \phi_t \cdot f^{\bar \phi^{\perp}}- \varepsilon^2 \bar \phi_{1t}\cdot f_1 \\
  &= (\bar \phi_{0t}+ \varepsilon \bar \phi_{1t}+ w_t)\cdot \Big(\varepsilon f_1- \langle \varepsilon f_1, \bar \phi_0+ \varepsilon \bar \phi_1+ w\rangle (\bar \phi_0+ \varepsilon \bar \phi_1+ w)\Big)-\varepsilon^2 \bar \phi_{1t}\cdot f_1 \\
  &= (\varepsilon \bar \phi_{1t}+ w_t)\cdot \Big(\varepsilon f_1- \langle \varepsilon f_1, \varepsilon \bar \phi_1+ w\rangle (\bar \phi_0+ \varepsilon \bar \phi_1+ w)\Big)-\varepsilon^2 \bar \phi_{1t}\cdot f_1 \\
  &=\varepsilon w_t\cdot f_1 - \langle \varepsilon f_1, \varepsilon \bar \phi_1+ w\rangle(\varepsilon \bar \phi_{1t}+ w_t)\cdot (\bar \phi_0+ \varepsilon \bar \phi_1+ w).
 \end{align*}
 Thus
 \begin{equation*}
     |\int_{\Sph^1} \bar \phi_t \cdot f^{\bar \phi^{\perp}}- \varepsilon^2 \bar \phi_{1t}\cdot f_1 dx|\lesssim \varepsilon^3 \|f_1\|_{L^2}+ \|\varepsilon f_1\|_{L^2} \varepsilon (\varepsilon^2+ \varepsilon \|\bar \phi_{1t}\|_{L^2})\lesssim \varepsilon^3,
 \end{equation*}
 for $\forall t\in [0, T]$.
 Therefore,
 \begin{equation*}
      |\int_0^T \int_{\Sph^1} \left( \bar \phi_t \cdot f^{\bar \phi^{\perp}}- \varepsilon^2 \bar \phi_{1t}\cdot f_1 \right)dx dt|\lesssim \varepsilon^3.
 \end{equation*}

In conclusion,
\begin{align}
     E(T)- E(0)&= -2\int_0^T \int_{\Sph^1} \bar \phi_t \cdot f^{\bar \phi^{\perp}} dx dt, \notag\\
     &= -2\varepsilon^2 \int_0^T \int_{\Sph^1} \bar \phi_{1t}\cdot f_1 dx dt+ O(\varepsilon^3) \label{es:strict:ET}\\
     &= -2\varepsilon^2 \int_0^T \int_{\Sph^1} \bar v_{t}\cdot g dx dt+ O(\varepsilon^3) \notag\\
     &= -2\pi \varepsilon^2+ O(\varepsilon^3). \notag
 \end{align}

 \end{proof}

\begin{remark}
It is noteworthy that  the energy is lower at time $T$ than at time $0$ since the geodesic is not a local minimiser (for the Dirichlet functional): this comes from the fact that the linearized system is $\Box v+ v=g$. On the other hand, if the geodesic were a local minimiser, then we would obtain a linearized system like $\Box v- v= g$, which would force the energy to increase.
\end{remark}

\noindent {\bf Step 2. Decrease of the energy near harmonic maps.}\\

Let $T= 2\pi$  and $f= \varepsilon_0 f_1$ as given in Proposition \ref{prop_decrea}. It is shown that the energy at time $T$ is strictly smaller than $2\pi$ if one starts from $\bar u[0]$ at time $0$; see Equation \eqref{decay-u0es}.
In this part, we perform a standard perturbation argument to show that for  initial states sufficiently close to $\bar u[0]$, the above designed control still decreases the energy  below the critical value $2\pi$.
\begin{equation*}
\begin{cases}
        \Box \phi= \left(|\phi_{t}|^2- |\phi_{ x}|^2\right) \phi+
     (\varepsilon_0 f_1)^{\phi^{\perp}}, \\
     (\phi, \phi_t)(0, x)= u[0](x).
     \end{cases}
\end{equation*}

By comparing the preceding equation with Equation \eqref{eq_full},
thanks to the continuous dependence property of the inhomogeneous wave maps equation, Lemma \ref{lem-conti-dep-inh}, the difference $w:= \phi- \bar \phi$ satisfies
\begin{gather*}
   \|(w_x, w_t, w)\|_{L^{\infty}_t L^2_x(D)}+  \|w_u\|_{L^{2}_u L^{\infty}_v\cap L^{\infty}_v L^{2}_u(D)}+  \|w_v\|_{L^{2}_v L^{\infty}_u\cap L^{\infty}_u L^{2}_v(D)}
   \leq C \lVert w[0]\lVert_{H^1\times L^2}.
\end{gather*}
In particular, there exists some effectively computable constant $\nu_0$ such that the energy estimate \eqref{es-per-2pi} holds provided the condition \eqref{cond-es-u0} is satisfied.  This finishes the proof of Proposition \ref{prop-pse-decay}, and further the proof of Theorem \ref{THM-globalexact}.

\end{proof}

\subsection{A remark on decreasing the energy around harmonic maps in small time.}
In this section, we continue the study of decreasing the energy of the system around harmonic maps. Recall that when the (spatial) controlled domain is the whole domain we are able to decrease the energy in an arbitrarily small time period, while if the (spatial) control domain is small the same  task can be fulfilled in a relatively large time $i.e.\; T=2\pi$.  It remains to understand the same problem for arbitrary control domain and in small time.

Let  $T>0$.  We restrict ourselves to the case where the control is small, which is equivalent to considering the control in the form
\begin{equation*}
    f(t, x)= \varepsilon f_1(t, x) \textrm{ with }  \|f_1\|_{L^{\infty}(0, T; L^2(\Sph^1))}\leq C_0
\end{equation*}
for some fixed positive constant $C_0$, where $f^1_1$ and $f^2_1$ are not necessarily zero.  Actually, for any given control such that $\|f\|_{L^{\infty}(0, T; L^2(\Sph^1))}\leq \varepsilon C_0$, it suffices to set
\begin{equation*}
    f_1(t, x):= \frac{1}{\varepsilon} f(t, x),
\end{equation*}
which satisfies
\begin{equation*}
    \|f_1\|_{L^{\infty}(0, T; L^2(\Sph^1))}\leq C_0.
\end{equation*}

Thanks to \eqref{eq:barphi1}, we know that  $\bar \phi_1$ satisfies
\begin{equation}\notag
     \begin{cases}
      \Box  \bar\phi^1_1+  \bar\phi^1_1+ 2\left(-(\sin{x}) ( \bar\phi^1_1)_x+ (\cos{x}) ( \bar\phi^2_1)_x+ f_1^1 \cos x+ f_1^2 \sin x\right) \cos{x}= f_1^1, \\
      \Box  \bar\phi^2_1+  \bar\phi^2_1+ 2\left(-(\sin{x})( \bar\phi^1_1)_x+ (\cos{x})( \bar\phi^2_1)_x+ f_1^1 \cos x+ f_1^2 \sin x\right) \sin{x}= f_1^2, \\
      \Box  \bar\phi^3_1+  \bar\phi^3_1= f^3_1,\\
      \bar \phi_1[0]= 0,
     \end{cases}
 \end{equation}
or, equivalently,
\begin{equation}\label{eq:linearbarphi1}
\begin{cases}
        -\bar \phi_{1tt}+ \bar \phi_{1xx}+ \bar \phi_1= -2 (\bar \phi_{0x}\cdot \bar \phi_{1x}) \bar \phi_0+ f_1- (f_1\cdot \bar \phi_0) \bar \phi_0,\\
        \bar \phi_1[0]= 0.
        \end{cases}
\end{equation}

 Moreover, according to the estimates between Equations \eqref{decay-u0es}--\eqref{es:strict:ET} (which remain the same for the general case when  $f^1_1$ and $f^2_1$ are non-zero functions), one has
 \begin{equation*}
     E(T)- E(0)= -2\varepsilon^2 \int_0^T \int_{\Sph^1} \bar \phi_{1t}\cdot \left(f_1-  (f_1\cdot \bar \phi_0) \bar \phi_0\right) dx dt+ O(\varepsilon^3).
 \end{equation*}

Let us  define
\begin{equation}
    \bar F(t):= \int_{\Sph^1} |\bar \phi_{1t}|^2+ |\bar \phi_{1x}|^2- |\bar \phi_1|^2 dx.
\end{equation}
Using Equation \eqref{eq:linearbarphi1}, one gets
\begin{align*}
    \frac{d}{dt} \bar F(t)&= 2 \int_{\Sph^1}  \bar \phi_{1t} \cdot \left( \bar \phi_{1tt}- \bar\phi_{1xx}- \bar \phi_1 \right) dx, \\
    &= 4\int_{\Sph^1} (\bar \phi_{0x}\cdot \bar \phi_{1x}) (\bar \phi_0\cdot \bar \phi_{1t})dx dt- 2\int_{\Sph^1}  ( \bar \phi_{1t}\cdot (f_1- (f_1\cdot \bar \phi_0) \bar \phi_0 ) )dx dt,
\end{align*}
which implies that
\begin{multline}
   \varepsilon^2 (\bar F(T)- \bar F(0)) = - 2\varepsilon^2 \int_0^T \int_{\Sph^1}  ( \bar \phi_{1t}\cdot (f_1- (f_1\cdot \bar \phi_0) \bar \phi_0 ) ) dx dt
   \\ + 4\varepsilon^2 \int_0^T \int_{\Sph^1} (\bar \phi_{0x}\cdot \bar \phi_{1x}) (\bar \phi_0\cdot \bar \phi_{1t})dx dt.
\end{multline}
We also notice that
\begin{align*}
    0&= \bar \phi \cdot \bar \phi_t,\\
    &= (\bar \phi_0+ \bar \varepsilon \phi_1+ w)\cdot(\bar \phi_{0t}+  \varepsilon \bar\phi_{1t}+ w_t), \\
    &= \varepsilon \bar \phi_0\cdot  \bar\phi_{1t}+  ( \varepsilon  \bar \phi_1+ w)\cdot \varepsilon \bar \phi_{1t}+  \bar \phi \cdot w_t.
\end{align*}
Hence
\begin{equation*}
\begin{array}{rcl}
\varepsilon \int_0^T \int_{\Sph^1} (\bar \phi_{0x}\cdot \bar \phi_{1x}) (\bar \phi_0\cdot \bar \phi_{1t})dx dt
  &=&  \int_0^T \int_{\Sph^1} (\bar \phi_{0x}\cdot \bar \phi_{1x}) (\bar \phi_0\cdot  \varepsilon\bar \phi_{1t})dx dt \\[4pt]
   % &= \int_0^T \int_{\Sph^1} (\bar \phi_{0x}\cdot \bar \phi_{1x}) (\bar \phi_0\cdot  (\bar \phi_t- \bar \phi_{0t}- w_t))dx dt, \\
   % &= -\int_0^T \int_{\Sph^1} (\bar \phi_{0x}\cdot \bar \phi_{1x}) (\bar \phi_0\cdot   w_t)dx dt+ \int_0^T \int_{\Sph^1} (\bar \phi_{0x}\cdot \bar \phi_{1x}) (\bar \phi_0\cdot  \bar \phi_t)dx dt , \\
   % &= -\int_0^T \int_{\Sph^1} (\bar \phi_{0x}\cdot \bar \phi_{1x}) (\bar \phi_0\cdot   w_t)dx dt+
   &=& - \int_0^T \int_{\Sph^1} (\bar \phi_{0x}\cdot \bar \phi_{1x}) \left( ( \varepsilon  \bar \phi_1+ w)\cdot \varepsilon \bar\phi_{1t}+  \bar \phi \cdot w_t\right) dx dt \\[4pt]
   &=& O(\varepsilon^2),
\end{array}
\end{equation*}
and
\begin{equation*}
      \varepsilon^2 (\bar F(T)- \bar F(0)) = - 2\varepsilon^2 \int_0^T \int_{\Sph^1}  ( \bar \phi_{1t}\cdot (f_1- (f_1\cdot \bar \phi_0) \bar \phi_0 ) ) dx dt+ O(\varepsilon^3).
\end{equation*}

The preceding equations imply that
\begin{equation}
    E(T)- E(0)=  \varepsilon^2 (\bar F(T)- \bar F(0))+  O(\varepsilon^3).
\end{equation}

 Therefore, it becomes a problem on decreasing the energy of $\bar \phi_1$ which is governed by a  linear controlled equation: Equation \eqref{eq:linearbarphi1}.  In the following proposition, a  ``non-trivial control" $f_1\in L^{\infty}(0, T; L^2(\Sph^1))$ infers to some control $f_1$ such that the unique solution of  \eqref{eq:linearbarphi1}  satisfies $\bar \phi_1[T]\neq 0$.
 \begin{prop}\label{prop:smalltimecontrol}
     {\it (i)}  Suppose that the controlled domain is $(-a, a)$ with $a< \frac{\pi}{2}$, then for any non-trivial control $f_1\in L^{\infty}(0, T; L^2(\Sph^1))$ with $T< (\pi/2)-a$, the unique solution of \eqref{eq:linearbarphi1} satisfies $\bar F(T)>0$.

    {\it (ii)}  Suppose that the controlled domain is $(-a, a)$ with $a> \frac{\pi}{2}$, then, for any $T>0$, there exists a  non-trivial control $f_1\in L^{\infty}(0, T; L^2(\Sph^1))$  such that the unique solution of \eqref{eq:linearbarphi1} satisfies $\bar F(T)<0$.
 \end{prop}

%{\color{magenta} The choices of $a$ is not optimal: maybe there exists some $a_0$ such that, for $a<a_0$  there is Property (i) for some positive $T$, while for $a>a_0$ there is Property (ii). If this is the case, then it ought to be related to the sharp constant for Poincare's inequality: actually, it should be $a= \pi/2$, the proof is simply based on  Fourier series.  }

 \begin{proof}[Proof of Proposition \ref{prop:smalltimecontrol}]

     (i) By the  ``non-trivial control"  assumption, we know that either $\bar \phi_1(T)\neq 0$, or  $\bar \phi_1(T)= 0$ and $\bar \phi_{1t}(T)\neq 0$.  If it is the latter case,  $i. e, \; \bar \phi_1(T)= 0$ and $\bar \phi_{1t}(T)\neq 0$, then we immediately get $\bar F(T)>0$.  It remains to consider the case where  $\bar \phi_1(T)\neq 0$.

    Thanks to the finite speed of propagation, we know that $\phi_1[t]$ remains to be zero in the light cone

    \begin{equation*}
    x\in (a+ t, 2\pi-a-t)=: D(t) \textrm{ for any } t\in \left[0, T\right].
    \end{equation*}
    Indeed, it suffices to consider the energy inside this light cone:
    \begin{equation*}
    E^l(t):= \int_{D(t)} |\bar \phi_{1t}|^2+ |\bar \phi_{1x}|^2+ |\bar \phi_{1}|^2 dx,
    \end{equation*}
    which, thanks to direct energy estimates, satisfies
    \begin{equation*}
        \frac{d}{dt} E^l(t)\lesssim E^l(t).
    \end{equation*}

Therefore, since $T< (\pi/2)-a$, there exists $\delta\in (0,\pi/2)$ such that $\bar \phi_1(T, x)= 0$ for every $x\in [(\pi/2)-\delta, (3\pi/2)+ \delta]$.
Let us define  $a_0:= \pi- 2\delta>0$. Note that
\begin{equation}
\label{ineqPoincare}
    \int_{0}^{a_0} |f'(x)|^2 dx\geq \frac{a_0^2}{\pi^2}\int_{0}^{a_0} |f(x)|^2 dx, \; \forall f(x)\in H^1_0(0, a_0).
\end{equation}
Indeed, every function $f(x)$ from $H^1_0(0, a_0)$ can be written as
\begin{equation*}
    f(x)= \sum_{n\in \mathbb{N}^*} f_n \sin \left(\frac{n\pi x}{a_0}\right).
\end{equation*}
One has
\begin{align*}
    \int_{0}^{a_0} |f(x)|^2 dx&= \sum_{n\in \mathbb{N}^*} f_n^2 \frac{a_0}{2},\\
    \int_{0}^{a_0} |f'(x)|^2 dx&= \sum_{n\in \mathbb{N}^*} f_n^2 \left(\frac{n\pi}{a_0}\right)^2 \frac{a_0}{2}.
\end{align*}
which gives \eqref{ineqPoincare}.

% Moreover, for any $x\in [0, \frac{\pi}{4}]$ there is
% \begin{equation*}
%     |\bar \phi^k_1(T, x)|^2= |\int_{x}^{\frac{\pi}{4}} \partial_x \bar \phi^k_1(T, s) ds |^2\leq |\frac{\pi}{4}- x| \|\bar \phi^k_{1x}(T, x)\|_{L^2(\Sph^1)}^2, \; \forall k\in \{1, 2, 3\}.
% \end{equation*}
% Thus
% \begin{equation*}
%  0 <   \|\bar \phi_1(T)\|_{L^2(\Sph^1)}^2\leq \frac{\pi^2}{16} \|\bar \phi_{1x}(T)\|_{L^2(\Sph^1)}^2,
% \end{equation*}
Note that the support of $ \bar \phi_1(T)$ is included in the interval $(-a_0/2, a_0/2)$, which is of length $a_0<\pi$. Together with \eqref{ineqPoincare}, this implies that
\begin{equation*}
    \bar F(T)\geq \left(1- \frac{a_0^2}{\pi^2}\right) \|\bar \phi_{1x}(T)\|_{L^2(\Sph^1)}^2>0.
\end{equation*}
This concludes the proof of Statement (i) in Proposition \ref{prop:smalltimecontrol}.

(ii) This part is based on an explicit construction.  Let $a_1:= a+ (\pi/2)$.
We define the $2\pi$-periodic function $\varphi_0(x)$ as
\begin{equation}
    \varphi_0(x):=
    \begin{cases}
      \cos (\frac{\pi x}{a_1}), \; \forall x\in (- \frac{a_1}{2},    \frac{a_1}{2}), \\
        0, \textrm{ elsewhere}.
    \end{cases}
\end{equation}
This function $\varphi_0$ satisfies
\begin{equation*}
    \int_{\Sph^1} ((\varphi_{0x})^2- (\varphi_{0})^2) (x) dx =
 \frac{a_1}{2}\left( \frac{\pi^2}{a_1^2}- 1 \right) < 0.
\end{equation*}
% \begin{equation}
%     \varphi_0(x):=
%     \begin{cases}
%       \left(\frac{7}{2}\right)^{\frac{1}{2}}  - x, \; \forall x\in (0,    \left(\frac{7}{2}\right)^{\frac{1}{2}}), \\
%         0, \; \forall x\in (0, \pi), \\
%         \varphi(-x), \textrm{ elsewhere}.
%     \end{cases}
% \end{equation}
% This equation satisfies
% \begin{equation*}
%     \int_{\Sph^1} ((\varphi_{0x})^2- (\varphi_{0})^2) (x) dx = 2 \int_0^{  \left(\frac{7}{2}\right)^{\frac{1}{2}}} 1-  \left(\left(\frac{7}{2}\right)^{\frac{1}{2}}  - x\right)^2 dx < 0.
% \end{equation*}
By a standard mollifying procedure, one obtains a smooth $2\pi$-periodic function $\varphi_{1}$, supported in $(-a, a)$, such that
\begin{equation*}
    \int_{\Sph^1} ((\varphi_{1x})^2- (\varphi_{1})^2 )(x) dx  < 0.
\end{equation*}

For any given $T>0$, we can find an explicit time-dependent function $b(t)$ of class $C^2$ such that
\begin{equation}
    b(0)=  b'(0)= 0, b(T)= 1,  b'(T)= 0.
\end{equation}
Now, we find an explicit trajectory $(\bar \phi, f_1)$ of the controlled equation \eqref{eq:linearbarphi1}: for any $t\in [0, T]$ and any $x\in \Sph^1$,
\begin{gather*}
    \phi_1^1(t, x)= \phi_1^2(t, x)= 0, \\
    \phi^3_1(t, x)= b(t) \varphi_1(x), \\
    f_1^1(t, x)= f_1^2(t, x)= 0, \\
    f^3_1(t, x)= \left(-(\bar \phi^3_1)_{tt}+ (\bar \phi^3_1)_{xx}+ \bar \phi^3_1\right)(t, x).
\end{gather*}
This trajectory, which starts from $(0,0)$ at time $0$,  is supported in $[0,T]\times(-a, a)$, and satisfies
\begin{equation*}
    \bar F(T)- \bar F(0)<0.
\end{equation*}
This concludes the proof of Statement (ii) in Proposition \ref{prop:smalltimecontrol}.
 \end{proof}

\section{Sharp control time: $\mathbb{S}^1$-target case}\label{sec-op-contime}

In this section we focus on the wave maps equation for $\mathbb{S}^1\rightarrow \mathbb{S}^1$. Since, for every  solution  $\phi$ of the wave equation \eqref{eq:inhomowavemaps},  $\phi(0)$ is homotopic to   $\phi(T)$, their degree (simply the rotation number) must   coincide. \\

\noindent {\bf (1)}  We first try to control between  states having {\it zero degree}. %Actually, this is not a necessary and sufficient condition for stabilization, as we shall see later on.
As usual we characterize the function $\phi$ by the polar coordinate, $$\phi(t, x)= (\cos{\theta}, \sin{\theta})^{T} \; \textrm{ with } \theta= \theta(t, x)\in \mathbb{R},$$
as well as the control term
$$f^{\phi^{\perp}}(t, x)= h(t, x) (-\sin{\theta}, \cos{\theta})^T \;  \textrm{ with $h(t, x)\in \mathbb{R}$ and supp } h \subset [0, T]\times  \omega.$$
Because $\phi(t, x)$ is continuous with respect to the variables $t$ and $x$, the function $\theta(t, x)$ can also chosen to be continuous. Meanwhile, since $\phi(t)$ has zero degree, we have
\begin{equation}
    \theta(t, 0)= \theta(t, 2\pi),\; \forall t\in [0, T].
\end{equation}

Direct calculation yields
\begin{gather*}
    \phi_t= (-\sin{\theta}, \cos{\theta})^T \theta_t, \;\;  \phi_x= (-\sin{\theta}, \cos{\theta})^T \theta_x,\\
     \phi_{tt}= (-\cos{\theta}, -\sin{\theta})^T (\theta_t)^2+ (-\sin{\theta}, \cos{\theta})^T \theta_{tt},\\
    \phi_{xx}= (-\cos{\theta}, -\sin{\theta})^T (\theta_x)^2+ (-\sin{\theta}, \cos{\theta})^T \theta_{xx}.
\end{gather*}
Thus,  the controlled wave maps equation  \eqref{eq:inhomowavemaps}  becomes the controlled linear wave equation on $\mathbb{T}= \mathbb{R}/2\pi \mathbb{Z}$:
\begin{equation}\label{wave_control}
    \Box\theta = h  \; \textrm{ with control }  h \textrm{ satisfying supp } h \subset [0,T]\times \omega.
\end{equation}
According to the classical control theory on 1-D wave equations, this system is exact controllable within any time $T>T_0$, where the value of $T_0$ is given by \eqref{Optimal-control-time} and is optimal for this controllability property.\\

\noindent {\bf (2)}  Next, we turn to consider   states with an arbitrary given degree $N\in \mathbb{Z}$, namely, the polar coordinate satisfies
\begin{equation*}
    \theta(t, 2\pi)= \theta(t, 0)+ 2\pi N, \;\forall t\in [0, T].
\end{equation*}
Under the polar coordinate, the wave maps control problem  turns out to be a boundary control problem for $\theta$ on the interval $[0, 2\pi]$:
\begin{equation}
\begin{cases}
        \Box\theta(t, x) = h(t, x), \;\; \textrm{ supp } h \subset [0,T]\times \omega,\\
        \theta(t, 2\pi)= \theta(t, 0)+ 2\pi N, \\
        \theta_x(t, 2\pi)= \theta_x(t, 0).
\end{cases}
\end{equation}
By considering instead the function,
\begin{equation*}
\bar \theta(t, x)= \theta(t, x)- Nx, \; \forall t\in [0, T],\; \forall x\in [0, 2\pi],
\end{equation*}
again, it satisfies the controlled wave equation \eqref{wave_control} in $\mathbb{T}$. Hence, the  system is exactly controllable between states having the same degree within every time $T> T_0$, where the value of $T_0$ is given by \eqref{Optimal-control-time} and is optimal for this controllability property.
This completes the proof of Theorem \ref{THM-optimaltime-S1}.

\appendix

\section{Proof of Proposition \ref{lem:wellclosegene} }\label{Sec:App:A}

\begin{proof}[Proof of Proposition \ref{lem:wellclosegene}]
    Recall the basic energy estimate 
    \begin{equation*}
       \frac{d}{dt} E(t)= -2 \int_{\mathbb{S}^1} \phi_t(t, x)\cdot F(t, \phi(t, \cdot), \phi_t(t, \cdot))(x)^{\phi(t, x)^{\perp}} dx. 
    \end{equation*}
    This implies that the function $b(t):= \left(E(t)\right)^{1/2}, \forall t\in D(\phi, \phi_t)$ satisfies 
    \begin{equation*}
        \dot b(t)\leq  C_B(b(t)), \forall t\in D(\phi, \phi_t).
    \end{equation*}
In the following, we only present the proof of Property (i), while  Property (ii) is a direct consequence of Property (i) and the preceding energy estimates.\\ 

First, we prove the first part of Property (i) concerning the well-posedness of the closed-loop system in a small time. Assume that $T_1= 0$. Define $T_0(R)$ as 
\begin{equation*}
    T_0(R):= \frac{R}{C_B(2R)+ 1},
\end{equation*}
where the constant $C_B(2R)$ is given in the condition ($\mathcal{P}1$).
Then, \textit{a priori}, if there is a  solution $(\phi, \phi_t)$ on $[0, T_0(R)]$, then it satisfies 
\begin{equation*}
    \|(\phi, \phi_t)(t, \cdot)\|_{H^1\times L^2}\leq 2R, \forall t\in [0, T_0(R)].
\end{equation*}
Otherwise one may select 
\begin{equation*}
    t_0:= \inf \{t:  \|(\phi, \phi_t)(t, \cdot)\|_{H^1\times L^2}= 2R\}< T_0(R).
\end{equation*}
By the  choice of $t_0$, we know that 
\begin{gather*}
    b(t)=  \|(\phi, \phi_t)(t, \cdot)\|_{H^1\times L^2}< 2R, \forall t\in [0, t_0], \\
     b(t_0)=   \|(\phi, \phi_t)(t_0, \cdot)\|_{H^1\times L^2}= 2R, \\
       \dot b(t)\leq C_B(b(t))\leq C_B(2R), \forall t\in [0, t_0].
\end{gather*}
Thus,
\begin{equation*}
    b(t_0)\leq b(0)+ t_0 C_B(2R)< 2R.
\end{equation*}
This leads to a contradiction. 

Let $T\in (0, T_0(R)]$ that will be chosen later on.   Define 
\begin{equation*}
    \mathcal{B}_T:= \{(\phi_0, \phi_1)\in C([0, T]; H^1\times L^2(\mathbb{S}^1; T\mathbb{S}^{k})): \|(\phi, \phi_t)\|_{\dot H^1\times L^2}\leq 2R\},
\end{equation*}
which can be regarded as a closed subset of the Banach space $C([0, T]; H^1\times L^2(\mathbb{S}^1; \mathbb{R}^{k+1}))$. 
We consider a map $\mathcal{S}_T$ from $\mathcal{B}_T$ to $C([0, T]; H^1\times L^2(\mathbb{S}^1; T\mathbb{S}^{k}))$  as follows:
\begin{align*}
  \mathcal{S}_T:   \mathcal{B}_T\rightarrow C([0, T]; H^1\times L^2(\mathbb{S}^1; T\mathbb{S}^k)), \\
  (\phi_0, \phi_1)\mapsto \mathcal{S}_T(\phi_0, \phi_1),
\end{align*}
where $\mathcal{S}_T(\phi_0, \phi_1)$ is the unique solution of the inhomogeneous wave maps equation \eqref{eq:cauchyinhomowavemaps} with  initial state $(g_0, g_1)$ and control term $f(t, x):= F(t, \phi_0(t, \cdot), \phi_1(t, \cdot))(x)$. 
Thanks to the choice of $T_0$, we know that 
\begin{equation*}
    \mathcal{S}_T(\phi_0, \phi_1)\in \mathcal{B}_T.
\end{equation*}
We also know that 
\begin{equation*}
    \|F(t, \phi_0(t, \cdot), \phi_1(t, \cdot))(x)\|_{L^2_x(\mathbb{S}^1)}\leq C_B(2R), \forall t\in [0, T].
\end{equation*}
It suffices to show that for a good choice of $T$ the map $\mathcal{S}_T$ is a contraction. Denote the region $[0, T]\times \mathbb{S}^1$ by $D$. Taking $(\phi_0, \phi_1)$ and $(\tilde \phi_0, \tilde \phi_1)$ from the set $\mathcal{B}_T$, we define $\phi$ and $\tilde\phi$ as the unique solutions of the inhomogeneous wave maps equations
\begin{gather*}
    \Box \phi= \left(|\phi_{t}|^2- |\phi_{x}|^2\right) \phi+  F_0^{\phi^{\perp}}, \\ 
    (\phi, \phi_t)(0, x)= (g_0, g_1)(x),\\
    F_0(t, x):= F(t, \phi_0(t, \cdot), \phi_1(t, \cdot))(x)
\end{gather*}
and
\begin{gather*}
    \Box \tilde\phi= \left(|\tilde\phi_{t}|^2- |\tilde\phi_{x}|^2\right) \tilde\phi+  \tilde F_0^{\tilde\phi^{\perp}},  \\
    (\tilde\phi, \tilde\phi_t)(0, x)= (g_0, g_1)(x),\\
    \tilde F_0(t, x):= F(t, \tilde \phi_0(t, \cdot), \tilde \phi_1(t, \cdot))(x).
\end{gather*}
In other words, 
\begin{equation*}
    (\phi, \phi_t):= \mathcal{S}_T((\phi_0, \phi_1)) \textrm{ and } (\tilde \phi, \tilde \phi_t):= \mathcal{S}_T((\tilde \phi_0, \tilde \phi_1)). 
\end{equation*}
Since 
\begin{equation}
    \lVert \phi[0]\lVert_{\dot{H}^1\times L^2}+ \lVert \tilde \phi[0]\lVert_{\dot{H}^1\times L^2}+ \lVert F_0\lVert_{L^2_{t, x}(D)}+ \lVert \tilde F_0\lVert_{L^2_{t, x}(D)}\leq  M= 2R+ 2\sqrt{T_0(R)} C_B(2R),
\end{equation}
according to Lemma \ref{lem-conti-dep-inh}, there exists some effectively computable constant $C_{cd}(R)$ only depending on the value of $R$, such that, 
\begin{align*}
   &\;\;\;\; \|\mathcal{S}_T((\phi_0, \phi_1))-  \mathcal{S}_T((\tilde \phi_0, \tilde \phi_1))\|_{C([0, T]; H^1\times L^2(\mathbb{S}^1; \mathbb{R}^{k+1}))}\\
    &\leq C_{cd}(R) \|F_0- \tilde F_0\|_{L^2_{t, x}(D)}, \\
    &\leq C_{cd}(R) K(2R) T^{\frac{1}{2}} \|(\phi_0, \phi_1)-  (\tilde \phi_0, \tilde \phi_1)\|_{C([0, T]; H^1\times L^2(\mathbb{S}^1; \mathbb{R}^{k+1}))}.
\end{align*}
Therefore, by choosing 
\begin{equation*}
   T= T(R):= \min \left\{T_0(R), \left(\frac{1}{2 C_{cd}(R) K(2R) }\right)^2\right\},
\end{equation*}
we conclude from Banach fixed point theorem that the map $\mathcal{S}_T$  admits  a  unique fixed point in $\mathcal{B}_T$. This function is indeed the unique solution of the Cauchy problem \eqref{eq:cauchyclosedloop} in $[0, T(R)]$.\\

Next, we show the second property of (i), the continuous dependence of the solutions of the closed-loop system.  Assume that $T_1= 0$. Suppose that $(\phi, \phi_t)$ is the unique solution of the Cauchy problem \eqref{eq:cauchyclosedloop} with initial state $(g_0, g_1)$ and $(\tilde \phi, \tilde \phi_t)$ is the unique solution of the Cauchy problem \eqref{eq:cauchyclosedloop} with initial state $(\tilde g_0, \tilde g_1)$. Thanks to the first part of Property (i), we have 
\begin{gather*}
    \|(\phi, \phi_t)(t, \cdot)\|_{\dot H^1\times L^2}\leq 2R, \forall t\in [0, T(R)], \\
    \|(\tilde\phi, \tilde\phi_t)(t, \cdot)\|_{\dot H^1\times L^2}\leq 2R, \forall t\in [0, T(R)]. 
\end{gather*}
Since
\begin{gather*}
    \lVert \phi[0]\lVert_{\dot{H}^1\times L^2}+ \lVert F(t, \phi_0(t, \cdot), \phi_1(t, \cdot))(x)\lVert_{L^2_{t, x}((0, T(R))\times \mathbb{S}^1)}\leq  \frac{M}{2}= R+ \sqrt{T_0(R)} C_B(2R),\\
       \lVert \tilde \phi[0]\lVert_{\dot{H}^1\times L^2}+  \lVert F(t, \tilde\phi_0(t, \cdot), \tilde\phi_1(t, \cdot))(x)\lVert_{L^2_{t, x}((0, T(R))\times \mathbb{S}^1)}\leq  \frac{M}{2}= R+ \sqrt{T_0(R)} C_B(2R),
\end{gather*}
thanks to the choice of $C_{cd}(R)$, there is 
\begin{align*}
  &\;\;\;\;  \|(\phi, \phi_t)- (\tilde \phi, \tilde \phi_t)\|_{C([0, T(R)]; H^1\times L^2(\mathbb{S}^1; \mathbb{R}^{k+1}))}, \\
    &\leq  C_{cd}(R) \left(\|(g_0, g_1)- (\tilde g_0, \tilde g_1)\|_{H^1\times L^2}+ \|F(t, \phi_0(t, \cdot), \phi_1(t, \cdot))(x)- F(t, \tilde \phi_0(t, \cdot), \tilde \phi_1(t, \cdot))(x)\|_{L^2_{t, x}((0, T(R))\times \mathbb{S}^1)} \right), \\
   & \leq C_{cd}(R) \left(\|(g_0, g_1)- (\tilde g_0, \tilde g_1)\|_{H^1\times L^2}+ K(2R)\|(\phi_0(t, \cdot), \phi_1(t, \cdot))- (\tilde \phi_0(t, \cdot), \tilde \phi_1(t, \cdot))\|_{L^2(0, T(R); H^1\times L^2(\mathbb{S}^1))} \right)\\
    & \leq C_{cd}(R) \left(\|(g_0, g_1)- (\tilde g_0, \tilde g_1)\|_{H^1\times L^2}+ \sqrt{T(R)}K(2R)\|(\phi_0, \phi_1)- (\tilde \phi_0, \tilde \phi_1)\|_{C([0, T(R)]; H^1\times L^2(\mathbb{S}^1))} \right)\\
     & \leq C_{cd}(R) \|(g_0, g_1)- (\tilde g_0, \tilde g_1)\|_{H^1\times L^2}+ \frac{1}{2}\|(\phi, \phi_t)- (\tilde \phi, \tilde \phi_t)\|_{C([0, T(R)]; H^1\times L^2(\mathbb{S}^1; \mathbb{R}^{k+1}))}. 
\end{align*}
Hence 
\begin{equation*}
    \|(\phi, \phi_t)- (\tilde \phi, \tilde \phi_t)\|_{C([0, T(R)]; H^1\times L^2(\mathbb{S}^1; \mathbb{R}^{k+1}))}\leq 2 C_{cd}(R) \|(g_0, g_1)- (\tilde g_0, \tilde g_1)\|_{H^1\times L^2}.
\end{equation*}

This completes the proof of Proposition \ref{lem:wellclosegene}.

\end{proof}

\section*{Acknowledgments}
Part of this work was done when Jean-Michel Coron was visiting Bernoulli Center at EPFL, and when Joachim Krieger was visiting Tsinghua University.   
We appreciate the hospitality and financial support of these institutions.   Shengquan Xiang is financially  supported by “The Fundamental Research Funds for the Central Universities, 7100604200, Peking University”

\bibliographystyle{plain}
\bibliography{main}

\end{document}